\documentclass[12pt,oneside]{amsart}

\usepackage[a4paper, margin=3cm]{geometry}
\usepackage{amsmath}
\usepackage{amsthm}
\usepackage{amsfonts}
\usepackage{amssymb}
\usepackage{mathabx}
\usepackage{mathtools}
\usepackage[normalem]{ulem}
\usepackage{tikz}

\usepackage{graphicx}
\usepackage{caption}
\usepackage{subcaption}
\usepackage[export]{adjustbox}

\usepackage[nodayofweek]{datetime}
\usepackage{hyperref}
\usepackage{hyphenat}

\usepackage{color}
\usepackage{transparent}

\usepackage{thm-restate} 

\definecolor{Green}{RGB}{30, 150, 30}

\newcommand{\sminus}{\mathbin{\setminus}}

\newtheorem{theorem}{Theorem}[section]
\newtheorem{lemma}[theorem]{Lemma}
\newtheorem{proposition}[theorem]{Proposition}
\newtheorem{corollary}[theorem]{Corollary}

\theoremstyle{definition}
\newtheorem{definition}[theorem]{Definition}
\newtheorem{example}[theorem]{Example}
\newtheorem{remark}[theorem]{Remark}

\declaretheorem[
  style=remark,
  name=Claim,
  within=theorem,
]{claim}

\newcommand{\sep}{\operatorname{Sep}}

\newcommand{\FO}{\mathcal{F}_0(\Sigma)} 
\newcommand{\F}{\mathcal{F}(\Sigma)} 
\newcommand{\C}{\mathcal{C}}
\newcommand{\ksep}{\mathcal{K}}

\newcommand{\mcg}{\operatorname{MCG}}
\newcommand{\pmcg}{\operatorname{PMCG}}

\newcommand{\G}{\mathcal{G}}

\newcommand{\diam}{\text{diam}}

\newcommand{\mc}[1]{\mathcal{#1}}
\newcommand{\mf}[1]{\mathfrak{#1}}
\newcommand*\Wit{\operatorname{Wit}}
\newcommand{\hierarchical}{hierarchical}

\newcommand{\Dis}{\operatorname{Dis}}
\newcommand{\PDis}{\operatorname{PDis}}

\newcommand{\tsh}[1]{\left\{\kern-.7ex\left\{#1\right\}\kern-.7ex\right\}}
\newcommand{\Tsh}[2]{\tsh{#2}_{#1}}
\newcommand{\ignore}[2]{\Tsh{#2}{#1}}

\makeatletter
\def\subsection{\@startsection{subsection}{2}
  \z@{.5\linespacing\@plus.7\linespacing}{.3\linespacing}
  {\normalfont\scshape}}
\makeatother

\title{Thickness and Relative Hyperbolicity for Graphs of Multicurves}

\author{Jacob Russell}

\address{Rice University, 6100 Main St, Houston, TX 77005, USA}
\email{jacob.russell@rice.edu}

\author{Kate M. Vokes}
\address{University of Luxembourg, 2 avenue de l'Universit\'e, 4365 Esch-sur-Alzette, Luxembourg}
\email{kate.vokes@uni.lu}

\begin{document}

\maketitle

\begin{abstract}
    We prove that any  graph of multicurves satisfying certain natural properties is either hyperbolic, relatively hyperbolic, or thick. Further, this geometric characterization is determined by the set of subsurfaces that intersect every vertex of the graph. This extends  previously established results for the pants graph and the separating curve graph to a broad family of graphs associated to surfaces.
\end{abstract}

\tableofcontents

\section{Introduction}

Given a connected, oriented, finite type surface~$S$, there are many graphs we can associate to~$S$ whose vertices represent curves or multicurves in~$S$.
Such graphs have been central to the study of Teichm\"uller spaces, geometric structures on 3\hyp{}manifolds, and mapping class groups.
A prominent example is the curve graph $\C(S)$, which has a vertex for every isotopy class of essential simple closed curves on~$S$ and an edge joining two vertices if the curves are disjoint.
A seminal result in this area is Masur and Minsky's proof that the curve graph is Gromov hyperbolic~\cite{mm1}. However, many other graphs naturally associated to surfaces are not Gromov hyperbolic.

 The object of this article is to give a classification of the coarse geometry of  the \emph{hierarchical} graphs of multicurves defined by the second author in~\cite{vokessep}.
The definition of a hierarchical graph of multicurves (see Definition~\ref{definition:hierarchical}) is in fact a list of properties that hold for many naturally defined graphs of multicurves, in particular, the curve graph, the pants graph, the separating curve graph, the non\hyp{}separating curve graph, and the cut system graph.
The motivation for the name ``hierarchical'' is that the second author proved in~\cite{vokessep} that these graphs are hierarchically hyperbolic spaces as defined by Behrstock, Hagen and Sisto~\cite{hhs1,hhs2}  (see Section~\ref{section:hierarchical graphs} for more detail).

In the present work, we sort hierarchical graphs of multicurves into the following classes.

\pagebreak[4]

\begin{theorem}\label{intro_theorem:relatively_hyperbolic_vs_thick}
 If $\G(S)$ is a hierarchical graph of multicurves on $S$, then $\G(S)$ is either:
 \begin{enumerate}
     \item hyperbolic,
     \item relatively hyperbolic, with each peripheral quasi\hyp{}isometric to a product of two infinite diameter hierarchical graphs of multicurves,
     \item thick of order~$1$,
     \item thick of order at most~$2$.
 \end{enumerate}
 Furthermore, which of the above cases $\G(S)$ fits into is determined by the set of subsurfaces that every vertex of $\G(S)$ intersects.
\end{theorem}

 The concept of a \emph{thick} metric space was introduced by Behrstock, Dru\c{t}u and Mosher, who showed that being thick of any order is an obstruction to a space being relatively hyperbolic \cite{Behrstock_Drutu_Mosher_Thickness}. Conceptually, thick metric spaces have highly intersecting non-hyperbolic parts, contrasting with relatively hyperbolic spaces where the non-hyperbolic regions are isolated. We give a more detailed review of thick metric spaces in Section~\ref{section:thick background}. Geometric classifications of thickness versus relative hyperbolicity,  akin to Theorem \ref{intro_theorem:relatively_hyperbolic_vs_thick}, have previously been shown for Artin groups \cite{Behrstock_Drutu_Mosher_Thickness}, Coxeter groups \cite{BHS_coxeter_thick_random}, $3$-manifold groups \cite{Behrstock_Drutu_Mosher_Thickness}, and free-by-cyclic groups \cite{Hagen_free-by-cyclic}.

Key to determining into which of the classes in Theorem~\ref{intro_theorem:relatively_hyperbolic_vs_thick} a graph of multicurves $\G(S)$ fits is to examine the set of \emph{witnesses} for~$\G(S)$, that is, the set of non-pants subsurfaces of $S$ that every vertex of $\G(S)$ must intersect. The heuristic for the classification in terms of witnesses is the following: if all witnesses for $\G(S)$ take up ``more than half'' of the surface~$S$, then $\G(S)$ is hyperbolic; if there are witnesses that take up ``less than half'' of~$S$, then $\G(S)$ is thick; and if the minimal witnesses of $\G(S)$ take up ``exactly half'' of~$S$, then $\G(S)$ is relatively hyperbolic. In practice, this heuristic needs to be expressed more carefully. We state the complete details of the classification of the geometry of $\G(S)$ in terms of its witnesses in Theorem~\ref{theorem:witnesses_determine_geoemetry}.

One subtlety is that instead of considering the witnesses that are minimal with respect to inclusion, we in fact need to restrict to witnesses whose complement is connected.
We will call a witness whose complement is connected a \emph{co-connected} witness.
In some special cases---if the surface $S$ has no punctures or no genus---we can now realize the above heuristic using Euler characteristic.

\begin{corollary}\label{intro_corollary:closed_surface_case}
Let $S$ be either a closed surface or a punctured sphere, and $\G(S)$ be a hierarchical graph of multicurves. Let $\chi_{\min}$ be the least negative Euler characteristic of a co-connected witness for~$\G(S)$.
\begin{itemize}
    \item If $|\chi_{\min}| >  \frac{1}{2}|\chi(S)|$, then $\G(S)$ is hyperbolic.
    \item If $|\chi_{\min}| = \frac{1}{2}| \chi(S)|$, then $\G(S)$ is hyperbolic relative to products of two infinite diameter hierarchical graphs of multicurves.
    \item If $|\chi_{\min}| < \frac{1}{2}|\chi(S)|$, then $\G(S)$ is thick of order~$1$, and hence not relatively hyperbolic.
\end{itemize}
\end{corollary}

If the surface $S$ has both punctures and genus, then this simple Euler characteristic criterion no longer works, and the statement of Theorem~\ref{theorem:witnesses_determine_geoemetry} will be more involved.
We give more discussion and examples of why graphs of multicurves for general surfaces can fail to satisfy Corollary~\ref{intro_corollary:closed_surface_case} in Section~\ref{section:applications}.

Our main result (Theorem~\ref{theorem:witnesses_determine_geoemetry}) completes the classification for all hierarchical graphs of multicurves, but this is not the first work to consider the question of when a graph of multicurves is hyperbolic, relatively hyperbolic, or thick.
A classification of this kind for the pants graphs was completed by Brock and Masur in~\cite{Brock_Masur_WP_Low_complexity}, following on from work of Behrstock, Dru\c{t}u and Mosher~\cite{Behrstock_Drutu_Mosher_Thickness}   and Brock and Farb~\cite{Brock_Farb_Rank}.
In fact, the pants graph gives us all the cases in Theorem~\ref{intro_theorem:relatively_hyperbolic_vs_thick}, depending on the surface~$S$.
The authors of the present article also previously completed the classification for the case of the separating curve graph in~\cite{russell_vokes}.  These previously known results are special cases of
Theorem~\ref{theorem:witnesses_determine_geoemetry}. 

A new (to our knowledge) application of the results in this paper is a classification of the geometry of graphs of non\hyp{}separating multicurves, which gives another example where all the possibilities of Theorem~\ref{intro_theorem:relatively_hyperbolic_vs_thick} are realized.
The first bullet point---the hyperbolic case---was previously proved by Hamenst\"adt~\cite{Hamenstadt_non-separating}.

\begin{restatable}{corollary}{introcor}
\label{intor_corollary:Non-separating_curve_complexes}
Let $S$ have genus~$g$ and $p$~punctures, and let $k \leq g$. If $\G(S)$ is any hierarchical graph of multicurves whose vertices are all the non-separating multicurves of $S$ with $k$ components, then we have the following.
\begin{itemize}
    \item $\G(S)$ is hyperbolic  when $k< g/2+1$.
    \item $\G(S)$ is hyperbolic relative to peripheral subsets that are quasi-isometric  to a product of two curve graphs when $k = g/2+1$ and $p=0$.
    \item $\G(S)$ is thick of order at most~$2$ when $k=g/2+1$ and $p>0$.
    \item $\G(S)$ is thick of order~$1$ when $k > g/2+1$.
\end{itemize}
\end{restatable}

\subsubsection*{Acknowledgments} The first author would like to thank IHES and Fanny Kassel  for their hospitality and financial support during a visit where much of the work on this project began.  The authors thank the anonymous referee for a careful reading of the paper and numerous useful comments.

\section{Preliminaries}

\subsection{Curves on surfaces }

Throughout, we will consider connected, oriented, finite type surfaces.
As the graphs of multicurves we will consider do not distinguish between a boundary component and a puncture, we may assume that the surfaces we consider are without boundary, and hence homeomorphic to a surface $S_{g,p}$ with genus $g$ and $p$ punctures, for some $g$ and~$p$.

By a \emph{curve} on a surface $S$ we mean an isotopy class of simple closed curves on~$S$.
Unless otherwise stated, all curves will be essential and non\hyp{}peripheral, that is, not bounding a disk or once punctured disk.

\emph{Subsurfaces} of $S$ will be assumed to be  closed subsets of $S$.
A subsurface $Y$ may hence have both punctures (coming only from the set of punctures of~$S$) and boundary components (curves of $S$ where $Y$ meets its complement in $S$).
We will use the notation $S_{g,p}^b$ for the homeomorphism class of subsurface with genus~$g$, $p$~punctures and $b$~boundary components. We define the \emph{complexity} of a subsurface $S_{g,p}^b$ to be  $\xi(S_{g,p}^b) = 3g-3+p+b$.  For a subsurface~$Y$, the notation $\partial Y$ refers to the set of boundary components of $Y$ and excludes  any punctures of $S$ that are in $Y$.
Subsurfaces here will be \emph{essential}, meaning that every boundary component of the subsurface is an essential, non\hyp{}peripheral curve of $S$.  We will use $D_p$ to denote the surface $S_{0,p}^1$, that is, a disk with $p$ punctures.  A subsurface is an \emph{annulus with punctures} if it homeomorphic to $S_{0,p}^2$ for some $p \geq 1$. 
As for curves, subsurfaces will be considered up to isotopy.

We say curves and/or subsurfaces are \emph{disjoint} if they have disjoint representatives. A \emph{multicurve}  on $S$ is a collection of pairwise disjoint, pairwise non-isotopic curves on~$S$. If a multicurve $\mu$ and a subsurface $Y$ are not disjoint,  we say $\mu$ \emph{intersects}~$Y$.  Abusing notation, if $\mu$ is a multicurve on~$S$, $S \sminus \mu$ will denote  the complement of a regular open neighborhood of~$\mu$.
Similarly, for a subsurface $Y$ of $S$, $S \sminus Y$ will denote the closure of the complement of~$Y$.  At times, we will use the notation $Y^c$ to denote $S \sminus Y$ for a subsurface $Y$.
We will call a subsurface $Y$ \emph{co-connected} if $S \sminus Y$ is connected. A multicurve $\mu$ is \emph{separating} if $S \sminus \mu$ is disconnected. If $Y$ and $Z$ are a pair of disjoint subsurfaces, then we use $Y \cup Z$ to denote the quotient of the disjoint union of $Y$ and $Z$ where a boundary component of $Y$ is glued to a boundary component of $Z$ if they are isotopic curves. In particular, $S \sminus (Y \cup Z)$ contains no annuli.
  The \emph{intersection number} of two multicurves $\mu$ and $\nu$ on $S$ is denoted $i(\mu,\nu)$ and is the minimal number of intersections between two representatives of $\mu$ and~$\nu$. 

The \emph{curve graph}, $\C(S)$, of a surface $S$ has a vertex for every curve on $S$ and an edge joining two vertices if they are disjoint. 
We make the standard modification for $S_{1,0}$, $S_{1,1}$ and $S_{0,4}$ by putting an edge between two vertices that intersect the minimum  number of times that a pair of simple closed curves on the surface can intersect.
For a subsurface $Y \subseteq S$, the curve graph $\C(Y)$ is defined similarly using the curves on $Y$ that are not isotopic to any curve in $\partial Y$. Note that $\C(S^b_{g,p})$ is identical to $\C(S_{g,p+b})$.
All graphs will be considered as metric spaces by declaring each edge to have length~1.

For every connected subsurface $Y$ of $S$ with $\C(Y)$ non-empty, Masur and Minsky defined a \emph{subsurface projection} map $\pi_Y \colon \C(S) \to 2^{\C(Y)}$. We recall a few properties and direct the reader to \cite[Section~2.3]{mm2} for full details.
For a set of curves~$A$ on~$S$, we define $\pi_Y(A)=\bigcup_{\alpha \in A}\pi_Y(\alpha)$. If $\mu$ is a multicurve on~$S$, then $\pi_Y(\mu)$ is empty if $\mu$ is disjoint from $Y$ and is a non-empty subset of diameter at most~3 if $\mu$ intersects~$Y$. If $\mu$ and $\nu$ are two multicurves on~$S$ that both intersect a subsurface~$Y$, then we define $d_Y(\mu,\nu)=\operatorname{diam}_{\C(Y)}(\pi_Y(\mu) \cup \pi_Y(\nu))$.

A \emph{graph of multicurves} on a surface $S = S_{g,p}$ is a non\hyp{}empty graph whose vertices are multicurves on~$S$.
If $\G(S)$ is a graph of multicurves on~$S$, then we say a connected subsurface $W \subseteq S$ is a \emph{witness} for $\G(S)$ if the interior of $W$ is not homeomorphic to $S_{0,3}$ and every vertex of $\G(S)$ intersects~$W$ (in other words, $W$ is a witness if every vertex of $\G(S)$ has non\hyp{}trivial subsurface projection to~$W$ \cite[Section~2]{mm2}).
If $\xi(S)\ge1$, the entire surface $S$ will always be a witness for~$\G(S)$.
We denote the set of  witnesses for $\G(S)$ by $\Wit\bigl(\mc{G}(S)\bigr)$.

\subsection{Mapping class groups} \label{section:mapping class groups}

We define the \emph{mapping class group} of~$S$, $\mcg(S)$, to be the group of isotopy classes of orientation\hyp{}preserving self\hyp{}homeomorphisms of~$S$. The action of $\mcg(S)$ on the set of curves on~$S$ induces an action of $\mcg(S)$ on $\C(S)$ by isometries. The \emph{pseudo-Anosov} elements of $\mcg(S)$ are precisely those that act loxodromically on the curve graph~$\C(S)$ \cite[Proposition~4.6]{mm1}. The \emph{pure mapping class group} of~$S$, $\pmcg(S)$, is the finite index subgroup of $\mcg(S)$ consisting of all elements of $\mcg(S)$ that fix each of the punctures of $S$.

For one of the main tools we will use  in this paper (Lemma~\ref{lemma:putman_trick}) we will need  explicit generating sets for $\mcg(S)$ and $\pmcg(S)$.
Depending on the situation, we will use different generating sets.

Our first generating set is a variant of Humphries' famous generating set for closed surfaces \cite{Humphries}. This generalization to the case of surfaces with punctures is  known to follow by applying the Birman exact sequence; see \cite[Section~4.4.4]{primer} for a description of this approach.

\begin{theorem}\label{theorem:humphries_generators}
Let $S= S_{g,p}$ with $g \geq 2$. Let $\Gamma = \{a_1,\dots, a_g, b_1,\dots, b_g, c\}$ be the set of curves on~$S$ shown in Figure~\ref{figure:Humphries_Generators_Moveable_Base} and let $\rho_1,\dots,\rho_p$ be the punctures of~$S$ if $p \geq 1$. Let $e_{1}, e_{2} \dots, e_{p-1}$ be a collection of disjoint curves on~$S$ so that there is an $i \in \{1.\dots, g\}$ for which the following holds for each $j \in \{1,\dots,p-1\}$,
\begin{itemize}
    \item $e_{j}$ intersects $b_{i}$ exactly once;
    \item $e_j$ is disjoint from and not equal to every curve in $\Gamma \sminus b_i$;
    \item for $j > 1$, $e_{j-1}$ and $e_{j}$ cobound a once punctured annulus containing $\rho_j$;
    \item $\rho_1$ and $\rho_p$ are in different components of $S \sminus (\Gamma \cup e_j)$.
\end{itemize}
If $X$ is the collection of Dehn twists around $a_1,\dots, a_g, b_1,\dots, b_g, c, e_1,\dots, e_{p-1}$, then $X$ generates $\pmcg(S)$. Furthermore, if  $\overline{X}$ is the union of $X$ with  a set of half twists that exchange every pair of punctures on $S$, then $\overline{X}$ generates $\mcg(S)$. \qed
\end{theorem}

\begin{figure}[ht]
    \centering
    \def\svgscale{.7}
    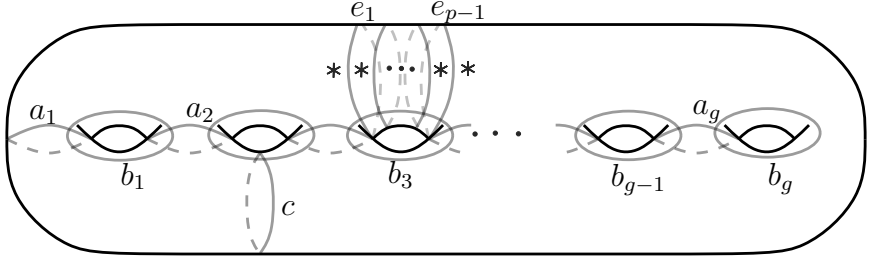
    \caption{The set of Dehn twists about the gray curves generates $\pmcg(S_{g,p})$. These Dehn twists plus half twists swapping punctures generates $\mcg(S_{g,p})$.}
    \label{figure:Humphries_Generators_Moveable_Base}
\end{figure}

A similar argument using the Birman exact sequence  applies to give the following generating set in the genus~$1$ case.

\begin{theorem}\label{theorem:gervais_generators}
Let $S = S_{1,p}$. Let $X$ be the set of Dehn twists about each curve on~$S$ shown in Figure~\ref{figure:Gervais_Generators} and let $\overline{X}$ be the union of~$X$ with a set of half twists that exchange every pair of punctures on $S$.  The set $X$ generates $\pmcg(S)$, while the set $\overline{X}$ generates~$\mcg(S)$. \qed
\end{theorem}

\begin{figure}[ht]
    \centering
    \def\svgscale{1.6}
    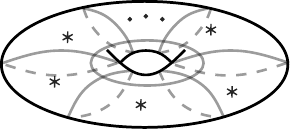
    \caption{Dehn twists about each of the gray curves generate $\pmcg(S_{1,p})$.}
    \label{figure:Gervais_Generators}
\end{figure}

Our final generating set is for surfaces with genus zero.  See \cite[Chapter 9]{primer} for details on this generating set.

\begin{theorem}[Braid generators for punctured sphere]\label{theorem:braid_generators}
Let $S = S_{0,p}$ with $p \geq 4$. Let $X$ be a set of elements of $\mcg(S)$ so that for every pair of punctures on $S$, $X$~includes a half twist exchanging them. Then the set $X$ generates $\mcg(S)$. \qed
\end{theorem}

For convenience, we name the sets of curves defining the generating Dehn twists in either Theorem~\ref{theorem:humphries_generators} or~\ref{theorem:gervais_generators}.

\begin{definition}[Standard generating curves]\label{definition:generating curves}We say a set of curves    on $S = S_{g,p}$ is a set of \emph{standard generating curves} if the set of Dehn twists about  those curves  gives the generating set for  $\pmcg(S)$ given in Theorem~\ref{theorem:humphries_generators} or Theorem~\ref{theorem:gervais_generators}.
\end{definition}

In Section~\ref{section:thick of order 2 case}, we will be making use of a map that fills in certain punctures of our surface~$S$.

\begin{definition}[The puncture\hyp{}filling map]
Let $S \cong S_{g,p}$ and let $\Pi$ be a subset of the set of punctures of~$S$.
Let $\Sigma \cong S_{g, p-\vert\Pi\vert}$ and let $T$ be the set of homotopy classes of curves on $\Sigma$ that bound a disk or once punctured disk. 
The puncture\hyp{}filling map $F_\Pi\colon \C(S) \rightarrow \C(\Sigma) \cup T$ is given by forgetting all punctures in the set~$\Pi$.
When the set $\Pi$ is understood, we will abbreviate $F_\Pi$ to~$F$.
\end{definition}

\begin{proposition}[Commutativity of puncture\hyp{}filling map]\label{proposition:capping commutes}
Each puncture-filling map $F \colon C(S) \to C(\Sigma) \cup T$ induces a surjective homomorphism $F_*\colon \pmcg(S) \rightarrow \pmcg(\Sigma)$ with the following property.
For all $f \in \pmcg(S)$ and any multicurve $c$ in $S$, we have $F_*(f)(F(c))=F(f(c))$. \qed
\end{proposition}

\subsection{Hierarchical graphs of multicurves} \label{section:hierarchical graphs}

The focus of this paper will be on graphs of multicurves that satisfy the following hypotheses. Recall that $\Wit\bigl(\mc{G}(S)\bigr)$ is the set of witnesses of $\G(S)$, that is, the set of non\hyp{}pants subsurfaces of~$S$ that intersect every vertex of $\G(S)$. 

\begin{definition}[Hierarchical graph of multicurves]\label{definition:hierarchical}
We call a graph of multicurves $\mc{G}(S)$  \emph{\hierarchical} if 
\begin{enumerate}
    \item $\mc{G}(S)$ is connected;
    \item the action of the mapping class group on the set of curves on~$S$ induces an action by graph automorphisms on $\mc{G}(S)$;
    \item there exists $R\geq 0$ such that any two adjacent vertices of $\mc{G}(S)$ intersect at most $R$ times; 
    \item $\Wit\bigl(\mc{G}(S)\bigr)$ contains no annuli.
\end{enumerate}
Note, given Item (2), Item (3) is equivalent to requiring that the action of $\mcg(S)$ on $\G(S)$ is cocompact.
\end{definition}

The curve graph and many other graphs associated to surfaces in the literature are hierarchical graphs of multicurves. Several specific examples of hierarchical graphs of multicurves and their witnesses  are discussed in Section~\ref{section:applications}.

The second author showed that every hierarchical graph of multicurves is a hierarchically hyperbolic space with respect to the Masur--Minsky subsurface projection maps~\cite{vokessep}. We direct the reader to \cite{hhs2,HHS_survey} for a complete definition of a hierarchically hyperbolic space and will instead only note the salient consequences of hierarchical hyperbolicity in the context of this paper.

\begin{theorem}[{\cite[Theorem 1.1]{vokessep}}]\label{theorem: hierarchial implies HHS}
If $\G(S)$ is a hierarchical graph of multicurves on the surface $S$, then $\G(S)$ is a hierarchically hyperbolic space. \qed
\end{theorem}

\begin{remark}\label{remark:disjoint_unions}
The hierarchically hyperbolic space structure for $\G(S)$ from Theorem \ref{theorem: hierarchial implies HHS} is built from the set of witnesses $\Wit(S)$ and the curve graphs for each witness. For technical reasons, the structure also includes disjoint unions of the subsurfaces in $\Wit(S)$ and their corresponding curve graphs, which are joins of the curve graphs of the connected components. The interested reader can find full details of the HHS structure for $\G(S)$ in \cite{vokessep}, but this distinction can safely  be ignored in the context of the present work.
\end{remark}

The most prominent consequence of hierarchical hyperbolicity is a distance formula in the same style as Masur and Minsky's distance formula for the mapping class group~\cite{mm2}.
For any witness $Y$ for $\G(S)$ and vertices $x, y \in \G(S)$, the subsurface projections of $x$ and $y$ to $\C(Y)$ are non\hyp{}empty, so the distance $d_Y(x,y)$ is defined.

\begin{theorem}[Distance formula; {\cite[Corollary 1.2]{vokessep},\cite[Theorem 4.5]{hhs2}}]
\label{theorem: distance formula} \ \\
Let $\G(S)$ be a hierarchical graph of multicurves on $S$ and $\mf{S} = \Wit(S)$. There exists $\sigma_0 >0$ such that for all $\sigma \geq \sigma_0$, there  are $K\geq 1$, $L\geq 0$ so
that for all $x,y\in \G(S)$,
$$ \frac{1}{K}  \sum_{Y\in \mf{S}}\ignore{d_Y(x,y)}{\sigma} - \frac{L}{K} \leq d_{\G}(x,y) \leq K \sum_{Y\in \mf{S}}\ignore{d_Y(x,y)}{\sigma} +L$$
\noindent where $ \ignore{N}{\sigma} = N$ if $N \geq \sigma$ and $0$ otherwise.  \qed
\end{theorem}

A consequence of the distance formula combined with Masur and Minsky's result that pseudo-Anosov maps act loxodromically on the curve graph~\cite[Proposition 4.6]{mm1}, is that any element of $\mcg(S)$ that  restricts to a pseudo-Anosov on a witness of $\G(S)$ has undistorted orbits in $\G(S)$.

\begin{corollary}[{\cite[Corollary 4.8]{russell_vokes}}]\label{corollary:pA_are_undistorted}
Let $S$ be a surface of positive complexity and $\G(S)$ be a hierarchical graph of multicurves  on $S$. Let $W$ be a witness for $\G(S)$  and $\phi \in\mcg(S)$ be a partial pseudo\hyp{}Anosov supported on $W$.
For all $x \in \G(S)$, the map $n \mapsto \phi^n(x)$ is a quasi-isometric embedding of $\mathbb{Z}$ into $\G(S)$. In particular, $\G(S)$ has infinite diameter.\qed
\end{corollary}

A critical step in Vokes' proof of Theorem \ref{theorem: hierarchial implies HHS} is the construction of a ``model space'' for a hierarchical graph of multicurves based only on the set of witnesses. The first step in this construction is defining a graph of multicurves corresponding to a collection of  connected subsurfaces of $S$.

\begin{definition} \label{definition: ksep}
Let $\mf{S}$ be a collection of connected  subsurfaces of $S = S_{g,p}$, with $\xi(W) \geq 1$ for all $W \in \mf{S}$.
Suppose further that $W \in \mf{S}$ and $W \subseteq Z$ implies $Z \in \mf{S}$.
If $\mf{S} = \emptyset$, define $\ksep_\mf{S}(S)$ to be a single point. Otherwise, define $\ksep_{\mf{S}}(S)$ to be the graph so that:
\begin{itemize}
    \item vertices are all multicurves $x$ on $S$ so that each component of $S \sminus x$ is \emph{not} an element of~$\mf{S}$;
    \item  two multicurves $x$ and $y$ are joined by an edge if either of the following conditions hold:
\begin{enumerate}
\item $x$ differs from~$y$ by either adding or removing a single curve; see Figure~\ref{figure: path1};
\item $x$ differs from $y$ by ``flipping" a curve in some subsurface of~$S$, that is, $y$ is obtained from $x$ by replacing a curve $\alpha$ by a curve~$\beta$, where $\alpha$ and $\beta$ are contained in the same component $Y_\alpha$ of $S \sminus (x\sminus\alpha)$ and are adjacent in~$\C(Y_\alpha$); see Figure~\ref{figure: path2}. 
\end{enumerate}
\end{itemize}
If $\G(S)$ is a graph of multicurves, define $\ksep_\G(S) = \ksep_\mf{S}(S)$ where $\mf{S} = \Wit\bigl(\G(S)\bigr)$. 
\end{definition}

\begin{figure}[h!]
\centering
\begin{subfigure}[b]{0.235\textwidth}
\centering
\includegraphics[width=\textwidth]{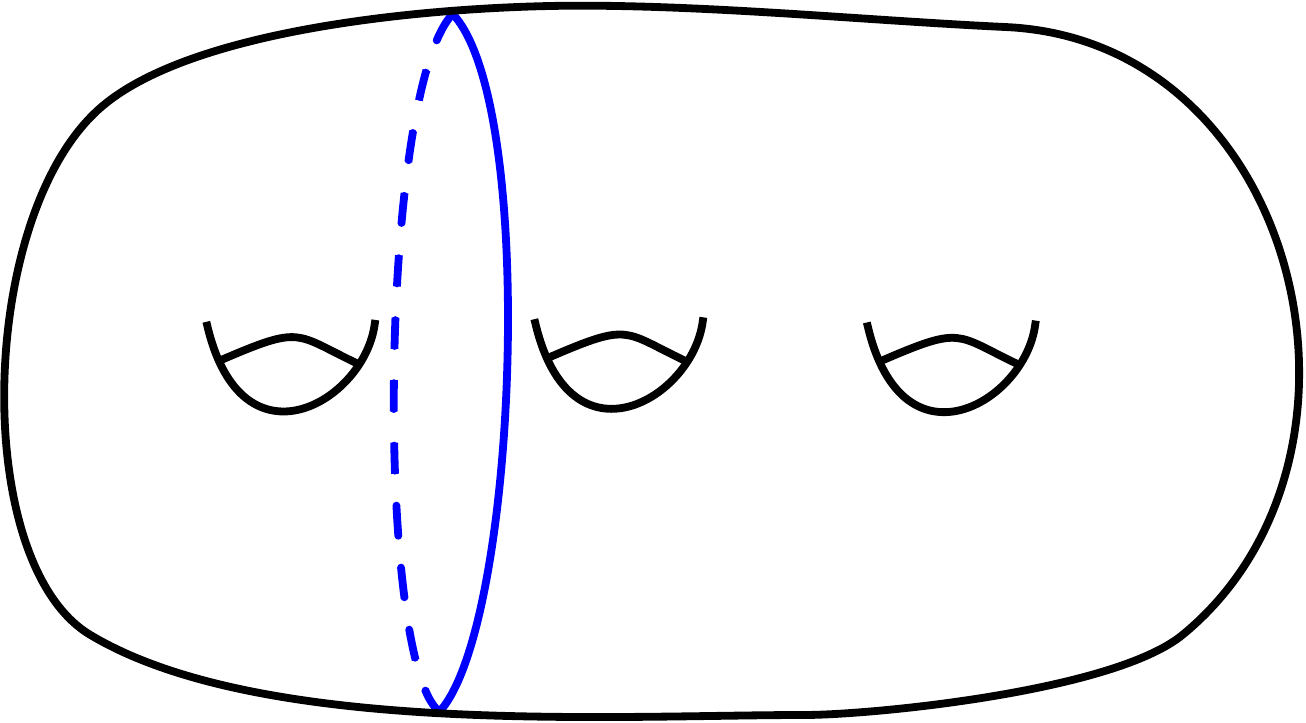}
\end{subfigure}
\begin{subfigure}[b]{0.235\textwidth}
\centering
\includegraphics[width=\textwidth]{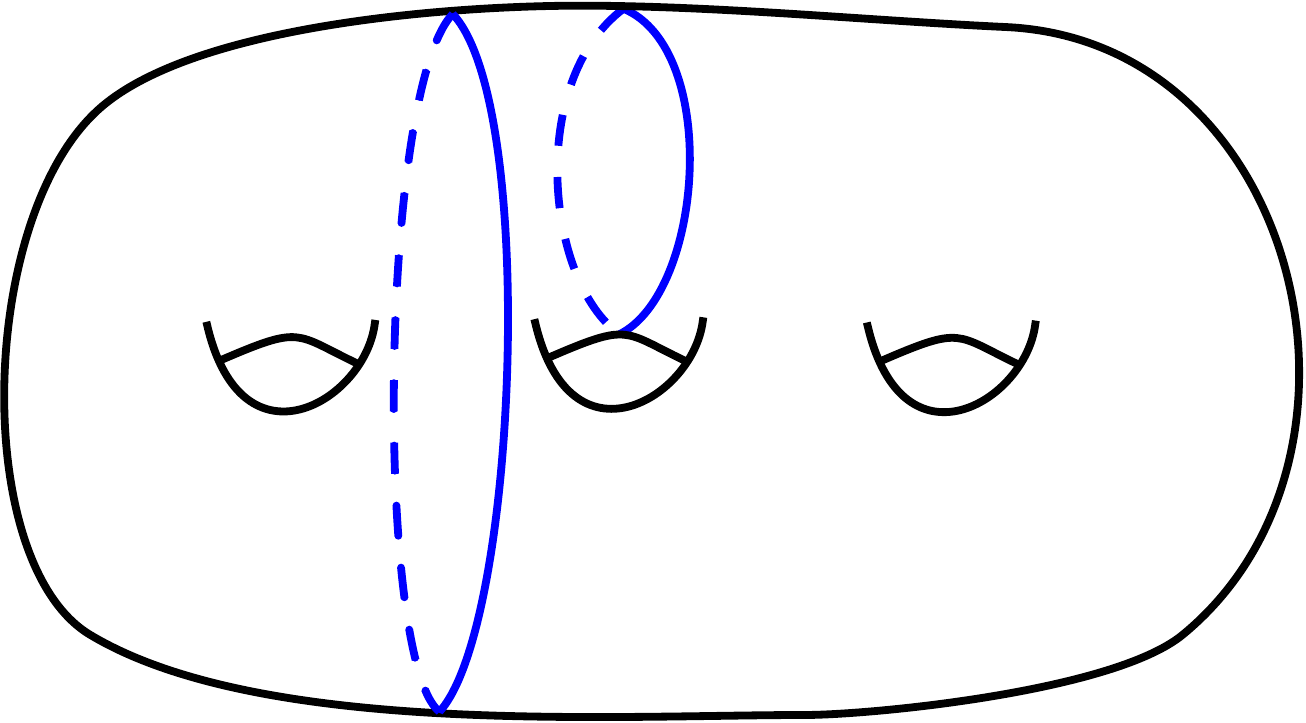}
\end{subfigure}
\begin{subfigure}[b]{0.235\textwidth}
\centering
\includegraphics[width=\textwidth]{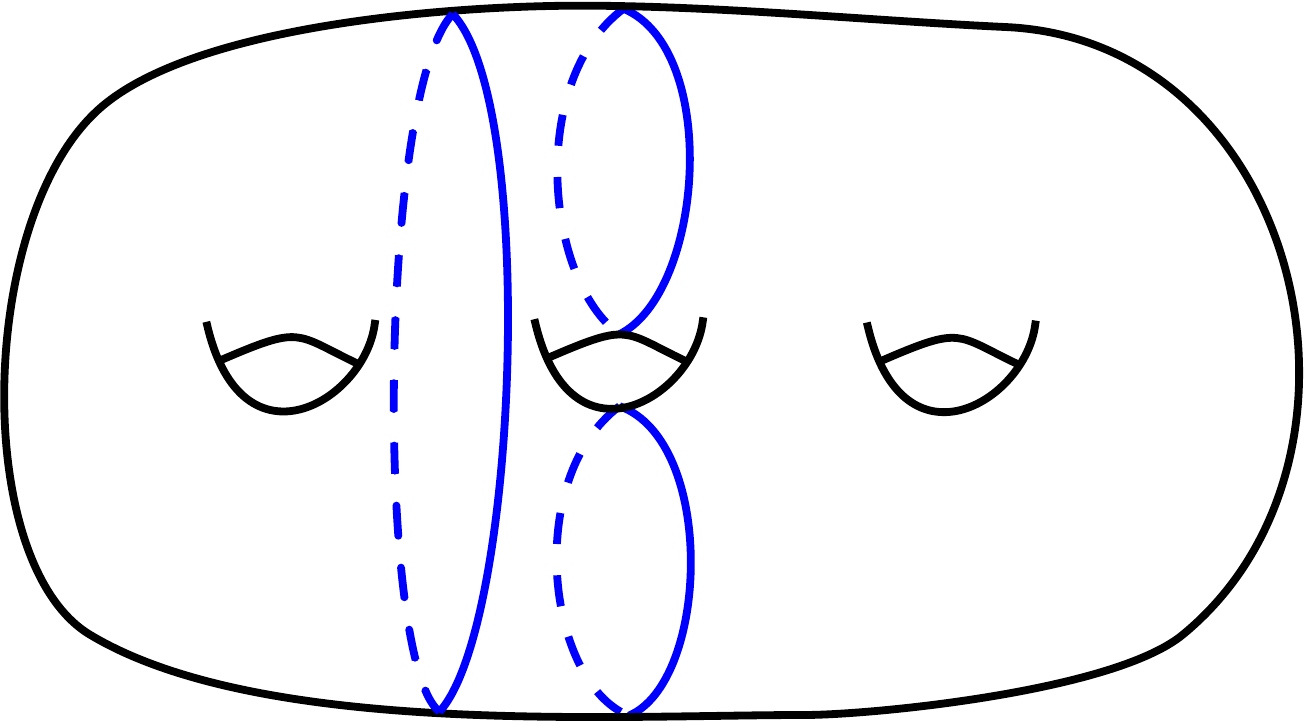}
\end{subfigure}
\begin{subfigure}[b]{0.235\textwidth}
\centering
\includegraphics[width=\textwidth]{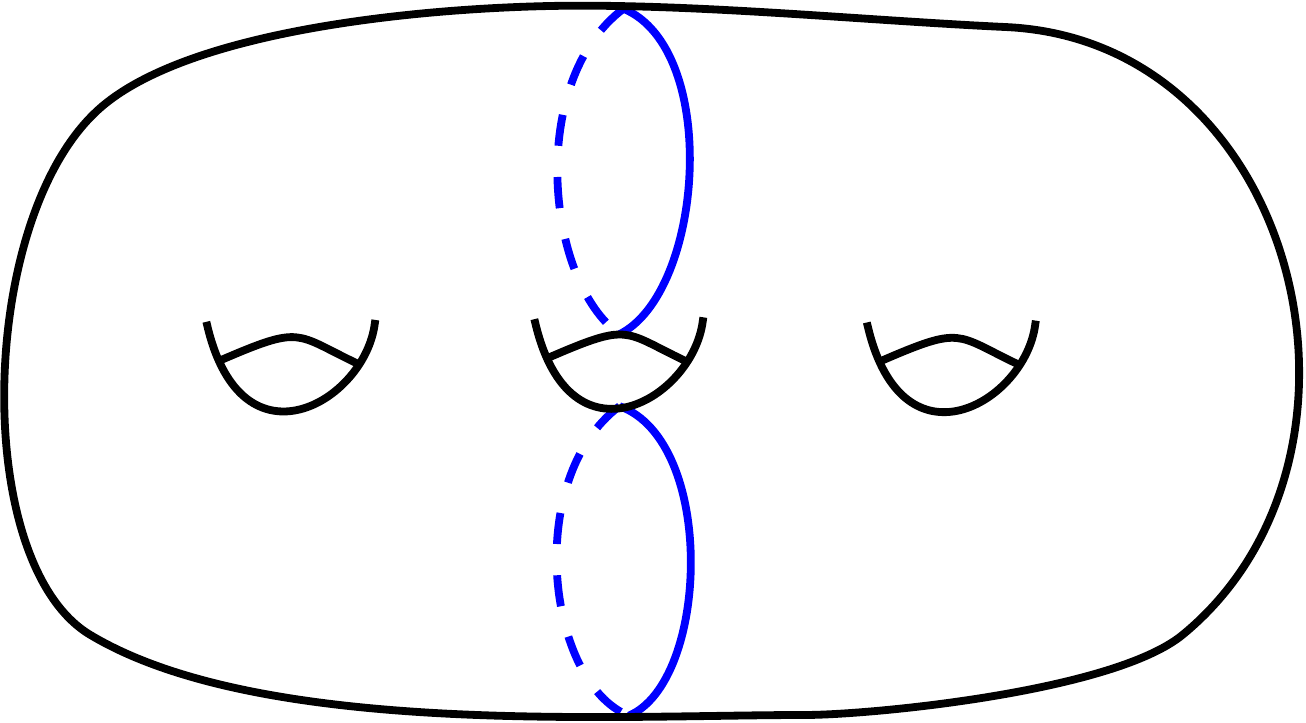}
\end{subfigure}
\caption{An example of a path in $\ksep_{\mf{S}}(S_3)$ given by adding and removing curves, where $\mf{S}$ is the set of subsurfaces whose complement does not contain genus. The set $\mathfrak S$ is the set of witnesses for the separating curve graph.}
\label{figure: path1}
\end{figure}

\begin{figure}[h!]
\centering
\begin{subfigure}[b]{0.25\textwidth}
\centering
\includegraphics[width=\textwidth]{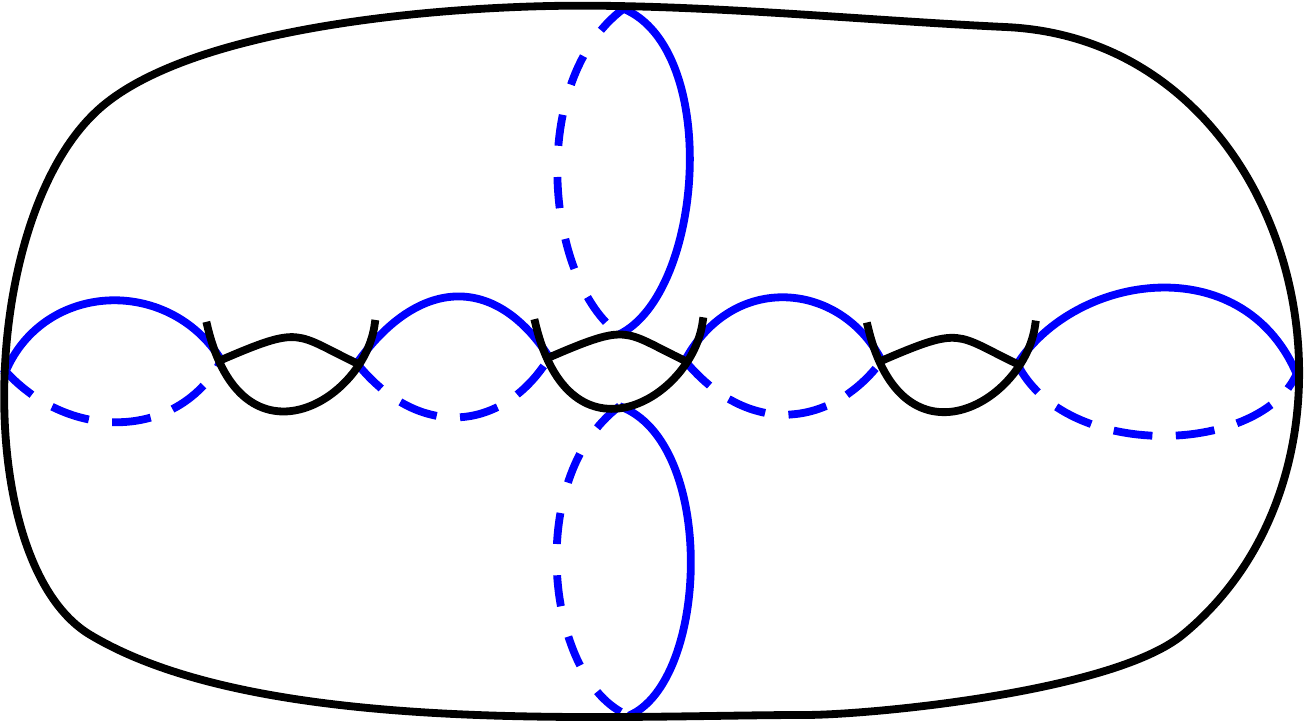}
\end{subfigure}
\qquad \qquad
\begin{subfigure}[b]{0.25\textwidth}
\centering
\includegraphics[width=\textwidth]{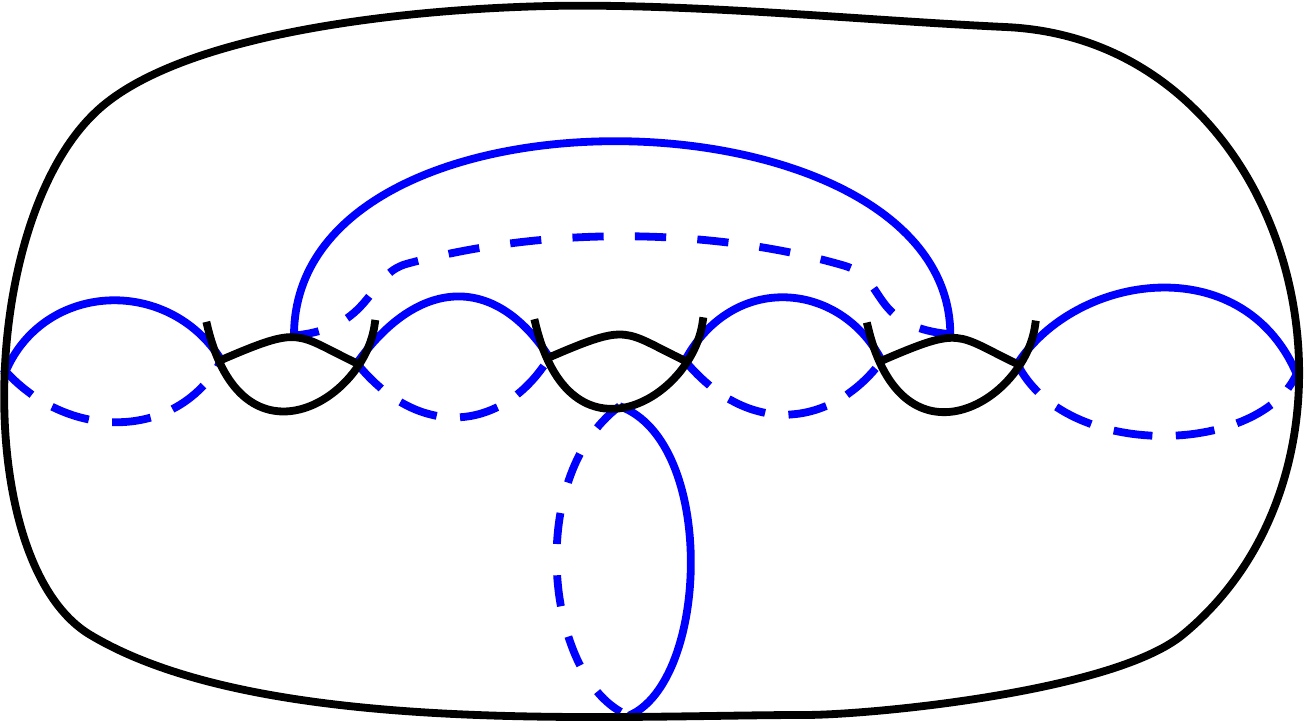}
\end{subfigure}
\caption{An example of a flip move in $\ksep_{\mf{S}}(S_3)$, where $\mf{S}$ is the set of subsurfaces whose complement does not contain genus.}
\label{figure: path2}
\end{figure}

 Since every vertex of $\G(S)$ will also be a vertex of $\ksep_\G(S)$, there is a natural inclusion $\G(S) \to \ksep_\G(S)$.
The following results tell us that $\ksep_\G(S)$ is a  quasi-isometric model for a hierarchical graph of multicurves that is determined solely by the witnesses of $\G(S)$. 

\begin{lemma}[{\cite[Section 3.1]{vokessep}}]\label{lemma:k is hierarchical}
Let $S$ be a surface with positive complexity. If $\G(S)$ is a hierarchical graph of multicurves on $S$, then $\ksep_\G(S)$ is hierarchical and $\Wit\bigl( \ksep_\G(S) \bigr) = \Wit\bigl(\G(S)\bigr)$. \qed
\end{lemma}

\begin{proposition}[{\cite[Proposition 4.1]{vokessep}}] \label{proposition: k qi g}
Let $\G(S)$ be a hierarchical graph of multicurves on~$S$. The inclusion map $\G(S) \to \ksep_\G(S)$ is a quasi-isometry. In particular, if $\G_1(S)$ and $\G_2(S)$ are two hierarchical graphs of multicurves and $ \Wit\bigl(\G_1(S)\bigr)   = \Wit\bigl(\G_2(S)\bigr)$, then $\G_1(S)$ and $\G_2(S)$ are quasi-isometric. \qed
\end{proposition}

More generally, the next lemma shows that $\ksep_\mf{S}(S)$ is hierarchical whenever $\mf{S}$ is $\mcg(S)$-invariant and satisfies the hypotheses of Definition \ref{definition: ksep}. This allows us to construct a hierarchical graph of multicurves with any possible set of witnesses. We will use this in Section \ref{section:applications} to construct graphs whose witnesses have specific properties.

\begin{lemma}\label{lemma:ksep_is_hierarchical}
Let $\mf{S}$ be a collection of subsurface of $S$ that satisfies the hypotheses of Definition \ref{definition: ksep}. If $\mf{S}$ is $\mcg(S)$\hyp{}invariant, then $\ksep_\mf{S}(S)$ is a hierarchical graph of multicurves.
\end{lemma}

\begin{proof}
Every vertex  of $\ksep_\mf{S}(S)$ can be connected to a pants decomposition of $S$ by a sequence of edges coming from adding curves. Hence, the fact that the pants graph is connected implies $\ksep_\mf{S}(S)$ is always connected \cite[Claim 3.3]{vokessep}. The $\mcg(S)$\hyp{}invariance of $\mf{S}$ implies that the action $\mcg(S)$ induces an action on $\ksep_\mf{S}(S)$ as well. The definition of edges in $\ksep_\mf{S}(S)$ ensures that vertices that are joined by an edge intersect at most twice. The fact that $\Wit(\ksep_\mf{S}(S))$ contains no annuli follows from the fact that $\mf{S}$ is required to contain no annuli and that $\Wit(\ksep_\mf{S}(S)) = \mf{S}$. 
\end{proof}

Every hierarchically hyperbolic space comes equipped with a system of sub-hierarchically hyperbolic spaces called product regions; see \cite[Section 5.2]{hhs2}. An advantage to working with the model space $\ksep_\mf{S}(S)$ is that these product regions can be described concretely without appealing to the broader theory of hierarchically hyperbolic spaces.

\begin{definition}[Product region of a multicurve]\label{definition:product_region}
Let $\mf{S}$ be the set of witnesses for some hierarchical graph of multicurves on $S$. If $m$ is a multicurve on $S$, define $P_\mf{S}(m) = \{ y \in \ksep_\mf{S}(S) : m\subseteq y \}$. We give $P_{\mf{S}}(m)$ the induced metric from $\ksep_\mf{S}(S)$.
\end{definition}

The following corollary to the distance formula provides the rationale for calling $P_\mf{S}(m)$ a ``product region''. In this corollary, if $\mf{S}$ is a collection of subsurfaces of $S$ and $Y$ is a subsurface of $S$, then $\mf{S}_Y = \{ Z \in \mf{S} : Z \subseteq Y\}$.

\begin{corollary}[{\cite[Corollary 4.11]{russell_vokes}}]\label{corollary:product regions}
Let $\mf{S}$ be the set of witnesses for some hierarchical graph of multicurves on~$S$. If $m$ is a multicurve on $S$ and $S \sminus m = Y_1 \sqcup \dots \sqcup Y_r$, then $P_\mf{S}(m)$ is quasi-isometric to $\prod_{i=1}^r \ksep_{\mf{S}_{Y_i}}(Y_i)$. Moreover, a factor $\ksep_{\mf{S}_{Y_i}}(Y_i)$  has infinite diameter if and only if $Y_i \in \mf{S}$. \qed
\end{corollary}

The final result we record on hierarchical graphs of multicurves is a result of  the second author that classifies when they are hyperbolic in terms of the set of witnesses. This paper is dedicated to understanding the coarse  geometry of a hierarchical graph of multicurves when Theorem \ref{theorem:hyperbolic_graphs_of_multicurves} does not apply.

\begin{theorem}[{\cite[Corollary 1.5]{vokessep}}]\label{theorem:hyperbolic_graphs_of_multicurves}
Let $\G(S)$ be a hierarchical graph of multicurves. The graph $\G(S)$ is hyperbolic if and only if $\Wit\bigl( \G(S) \bigr)$ does not contain a pair of disjoint subsurfaces. \qed
\end{theorem}

\subsection{Thick metric spaces and connectivity of graphs} \label{section:thick background}

 Behrstock, Dru\c{t}u, and Mosher introduced thick metric spaces to give a geometric obstruction to relative hyperbolicity \cite{Behrstock_Drutu_Mosher_Thickness}.

\begin{definition}[Thick metric space]
A metric space $X$ is \emph{thick of order~0} if at least one of its asymptotic cones has no cut points and every point~in $X$ is uniformly close to a bi-infinite uniform quasi-geodesic. In particular,  $X$~will be thick of order~0 if $X$~is quasi-isometric to a product of two infinite diameter metric spaces, contains a bi-infinite quasi-geodesic, and admits a cocompact group action.

A metric space $X$ is \emph{thick of order at most~$n$} if there exists a constant $C\geq 0$ and a collection of subsets $\{P_\alpha\}_{\alpha \in I}$ such that the following hold.

\begin{itemize}
    \item (Thickness) Each $P_\alpha$ is thick of order at most~$n-1$.
    \item (Coarsely Covering)  The space $X$ is contained in the $C$\hyp{}neighborhood of $\bigcup_{\alpha \in I} P_\alpha$.
    \item (Thick Chaining) For any $P_\alpha$ and $P_{\alpha'}$ there exists a sequence $$P_\alpha = P_0, P_1, \dots, P_k = P_{\alpha'}$$ such that $\mc{N}_C(P_i) \cap \mc{N}_C(P_{i+1}) $ has infinite diameter for all $0\leq i \leq k-1$.
\end{itemize}
\end{definition}

\begin{theorem}[{\cite[Corollary 7.2]{Behrstock_Drutu_Mosher_Thickness}}]
If a metric space $X$ is thick of any order, then $X$ is not relatively hyperbolic. \qed
\end{theorem}

Since the order of thickness is a quasi-isometry invariant \cite[Remark 7.2]{Behrstock_Drutu_Mosher_Thickness}, when studying the thickness of a hierarchical graph of multicurve $\G(S)$, we can instead study the quasi-isometric model space $\ksep_\G(S)$.

Proving the thick chaining condition is often the most challenging part of proving a metric space is thick of order at most~$n$. A common approach to surmount this problem is to find a graph whose vertices correspond to the thick of order at most $n-1$ subsets of~$X$ and whose edges correspond to the infinite diameter intersection of regular neighborhoods of those subsets. In this way, the question of thickly chaining the subsets  reduces to determining the connectivity of this graph.

In the context of hierarchical graphs of multicurves, Corollary \ref{corollary:product regions} provides natural thick of order~$0$ subsets---those product regions $P_\mf{S}(m)$ where at least two components of $S \sminus m$ are witnesses for our graph. In particular, the vertices of the following graph will always correspond to thick of order~$0$ subsets of $\ksep_\G(S)$ when $Z,W \in \Wit\bigl(\G(S)\bigr)$.

\begin{definition}
Let $\G(S)$ be a hierarchical graph of multicurves on $S$, and let $Z,W$ be a pair of disjoint witnesses for $\G(S)$. Define $\Dis(Z,W)$ to be the graph with the following vertices and edges.
\begin{itemize}
    \item Vertices: All multicurves $m$ such that $S \sminus m$ has at least two components $U$ and $V$ and there exists $f \in \mcg(S)$ such that $f(Z) \subseteq U$ and $f(W) \subseteq V$. 
    \item Edges: Two vertices $m$ and $n$ are joined by an edge if $m$ can be obtained from $n$ by adding or removing a curve.
\end{itemize}
\end{definition}

The graph $\Dis(Z,W)$ will guide the chaining together of the product regions corresponding to its vertices to make thick of order at most~1 subsets of $\ksep_\mf{S}(S)$.

\begin{lemma}\label{lemma:connected_components_are_thick_of_order_1}
Let $\G(S)$ be a hierarchical graph of multicurves and $Z,W \in \mf{S}=\Wit(\G(S))$  with $Z,W$ disjoint and $Z^c,W^c$ connected. If $\Omega$ is a connected subset of $\Dis(Z,W)$, then $\bigcup_{m \in \Omega} P_\mf{S}(m)$ is a thick of  order at most~1 subset of $\ksep_\mf{S}(S)$. In particular, if $\Dis(Z, W)$ is connected, then $\G(S)$ is thick of order at most~1.
\end{lemma}

\begin{proof}
We first show that the $P_\mf{S}(m)$ are thick of order~0.
By Corollary~\ref{corollary:product regions},  the product region $P_{\mf{S}}(m)$ is quasi-isometric to $\prod_{i=1}^r \ksep_{\mf{S}_{Y_i}}(Y_i)$, where $Y_1,\dots,Y_r$ are the components of $S \sminus m$. Each $\ksep_{\mf{S}_{Y_i}}(Y_i)$ admits a cocompact group action of $\mcg(Y_i)$ and at least two of the factors have bi-infinite quasi-geodesics because $S \sminus m$ contains a pair of disjoint witnesses. Therefore, $P_{\mf{S}}(m)$ is thick of order~0 for each $m \in \Omega$.

We now show that any two product regions for multicurves in $\Omega$ can be thickly chained together.
Let  $m, m' \in \Omega$. First assume $m$ and $m'$ are connected by an edge in $\Dis(Z,W)$. Thus, without loss of generality, $m \subseteq m'$ and we have $P_{\mf{S}}(m') = P_{\mf{S}}(m) \cap P_{\mf{S}}(m')$. Since $P_{\mf{S}}(m')$ is infinite diameter by Corollary~\ref{corollary:product regions}, we have $\diam( P_{\mf{S}}(m) \cap P_{\mf{S}}(m')) = \infty$. Now, assume $m$ and $m'$ are not joined by an edge in $\Dis(Z,W)$.  Since $\Omega$ is connected, there exists a path in $\Omega$ connecting $m$ and $m'$ with vertices $m = m_1, \dots, m_k = m'$. By the previous case, $\diam(P_\mf{S}(m_i) \cap P_{\mf{S}}(m_{i+1})) = \infty$ for all $i \in \{1,\dots, k-1\}$. Thus, $\bigcup_{m \in \Omega} P_\mf{S}(m)$ is a thick of order at most~1 subset of $\ksep_\mf{S}(S)$.

Now, assume $\Omega = \Dis(Z,W)$. Since the action of $\mcg(S)$ on $\ksep_{\mf{S}}(S)$ is cobounded, every vertex of $\ksep_{\mf{S}}(S)$ is at uniformly bounded distance from  a vertex of $P_\mf{S}(m)$ for some $m \in \Omega$.
Hence, $\ksep_{\mf{S}}(S)$ is thick of order at most~1.
Since the order of thickness is a quasi\hyp{}isometry invariant, this implies that $\G(S)$ is also thick of order at most~1.
\end{proof}

We shall see that in most cases where a hierarchical graph of multicurves $\G(S)$ is not hyperbolic or relatively hyperbolic, we can find disjoint, co-connected witnesses $Z$ and $W$ so that $\Dis(Z,W)$, or a sufficiently large subgraph, is connected, and hence $\G(S)$ is thick of order at most~1.

 Our main tool for understanding the connectivity of $\Dis(Z,W)$ is the following trick of Putman.

\begin{lemma}[{\cite[Lemma~2.1]{putman}}]\label{lemma:putman_trick}
Let $\G$ be a simplicial graph and suppose the group $G$ acts on $\G$ by simplicial automorphisms. Fix a vertex $v_0 \in \G$ and a generating set $X$ for $G$. Suppose that:
\begin{itemize}
    \item for all vertices $v \in \G$, the orbit $G \cdot v_0$ intersects the connected component of $\G$ containing $v$;
    \item for all $g \in X \sqcup X^{-1}$, $v_0$ is connected to $g\cdot v_0$ in $\G$.
\end{itemize}
Then, the graph $\G$ is connected. \qed
\end{lemma}

When we apply Lemma \ref{lemma:putman_trick}, $\G$ will be a graph of multicurves, usually $\Dis(Z,W)$ or a subgraph, and the group $G$ will be either the mapping class group or the pure mapping class group. In each case, we will use one of the generating sets for $\mcg(S)$ or $\pmcg(S)$ given in Section~\ref{section:mapping class groups}, arranging the generating curves to achieve a convenient intersection pattern with our chosen base vertex.

We will finish our preliminaries by combining several results from the literature to note that a hierarchical graph of multicurves can never be thick of order~$0$. This involves using the orbit of a pseudo-Anosov element of the mapping class group to produce a \emph{Morse quasi\hyp{}geodesic} in the graph, that is, a quasi\hyp{}geodesic $\gamma$ which is fellow\hyp{}traveled by every quasi\hyp{}geodesic whose endpoints lie in $\gamma$. We direct the reader to \cite[Section~1.3]{DMS_divergence} for details on Morse geodesics. The desired statement will then be a consequence of a result of Dru\c{t}u, Mozes and Sapir.

\begin{theorem}
If $\G(S)$ is a hierarchical graph of multicurves of $S= S_{g,p}$, then $\G(S)$ is not thick of order~$0$.
\end{theorem}

\begin{proof}
By \cite[Proposition 1]{DMS_Corrigendum}, if a metric space  contains a bi-infinite Morse quasi-geodesic, then all of its asymptotic cones have a cut-point. Since  such a space cannot be thick of order~$0$,
it suffices to find a bi-infinite Morse quasi-geodesic in $\G(S)$.
By applying \cite[Theorem 4.4 and Lemma 4.3]{ABD} to the hierarchically hyperbolic structure for~$\G(S)$, we have that if $\gamma$ is a quasi-geodesic in $\G(S)$ and  if there exists $D\geq 0$  so that for every $W \neq S$ in $\Wit\bigl( \G(S) \bigr)$ we have $\diam(\pi_W(\gamma))\leq D$, then $\gamma$ is Morse (\emph{stable} in the terminology of \cite{ABD}). We construct such a quasi-geodesic using the orbit of a pseudo-Anosov element of $\mcg(S)$.

Let $\phi$ be a pseudo\hyp{}Anosov element of $\mcg(S)$ and $x$ be a vertex of $\G(S)$.
By Corollary~\ref{corollary:pA_are_undistorted}, $\langle \phi \rangle \cdot x$ is a quasi\hyp{}geodesic in~$\G(S)$.  By \cite[Proposition 4.6]{mm1}, $ \pi_S( \langle \phi \rangle \cdot x)$ is also a quasi-geodesic in $\C(S)$ with constants depending on $\phi$. Since $\C(S)$ is hyperbolic, there exists $R \geq 0$ (depending on $\phi$) so that for each $n,m \in \mathbb{Z}$, any $\C(S)$-geodesic from $\pi_S(\phi^n(x))$ to $\pi_S(\phi^m(x))$ is contained in the $R$-neighborhood of the set of points $$ \bigcup_{n \leq \ell \leq m} \pi_S(\phi^{\ell}(x)). $$

For a witness $W \in \Wit(\G(S))$, let $N_W$ be the set of integers $n$ so that $\pi_S(\phi^n(x))$ intersects the $(R+2)$-neighborhood of  $\partial W$ in $\C(S)$. If $N_W$ is empty, then for any integers $n,m$ any $\C(S)$-geodesic from $\pi_S(\phi^n(x))$ to $\pi_S(\phi^m(x))$ does not intersect the $1$-neighborhood of $\partial W$ in $\C(S)$. Hence, Masur and Minsky's  bounded geodesic image theorem \cite[Theorem~3.1]{mm1} says there exists $M\geq0$ depending only on $S$ so that 
$$ \diam\left( \pi_W( \langle \phi \rangle \cdot x)  \right)\leq M.$$

If instead $N_W$ is non-empty, let $\min_W$ and $\max_W$ be its minimum and maximum values respectively---since $ \pi_S( \langle \phi \rangle \cdot x)$ is a quasi-geodesic, such a maximum and minimum must exist. If $n,m <\min_W$ or $n,m>\max_W $, then any $\C(S)$-geodesic from $\pi_S(\phi^n(x))$ to $\pi_S(\phi^m(x))$ does not intersect the $1$-neighborhood of $\partial W$ in $\C(S)$. Thus,
the  bounded geodesic image theorem says
$$ \diam\left( \bigcup_{n < \min_W} \pi_W(\phi^n(x))  \right)\leq M \text{ and }  \diam\left( \bigcup_{n > \max_W} \pi_W(\phi^n(x))  \right)\leq M. $$
Hence 
$$\diam( \pi_W(\langle \phi \rangle \cdot x )) \leq \diam\left( \bigcup_{\min_W-1 \leq n \leq \max_W+1} \pi_W(\phi^{n}(x))\right) + 2M.$$ The difference between $\min_W$ and $\max_W$ is bounded by a constant depending on $\phi$  (but not $W$) because  $ \pi_S(\langle \phi \rangle \cdot x)$ is a bi-infinite quasi\hyp{}geodesic in~$\C(S)$ with constants depending on $\phi$. Hence, the diameter of $ \{\phi^n(x): \min_W <n < \max_W\}$ in $\G(S)$ depends on $\phi$, but not $W$. Since the subsurface projection maps are uniformly coarsely Lipschitz, this plus the above  inequality implies $\diam( \pi_W(\langle \phi (x) \rangle ))$ is bounded above by a constant independent of $W$. Hence $\langle \phi \rangle \cdot x$ is a bi-infinite Morse quasi-geodesic in $\G(S)$ as desired.
\end{proof}

\subsection{Statement of main theorem}

The main result of this paper is the following classification of the hyperbolicity, relative hyperbolicity, and thickness of a hierarchical graph of multicurves in terms of its set of witnesses.
 Recall that an annulus (or disk) with punctures is a subsurface homeomorphic to $S_{0,p}^2$ (or $S^1_{0,p}$) where $p$ is at least~1.

\begin{theorem}\label{theorem:witnesses_determine_geoemetry}
Let $\G(S)$ be a hierarchical graph of multicurves on $S = S_{g,p}$.
\begin{enumerate}
    \item \label{item:main_thm_hyp} $\G(S)$ is hyperbolic if and only if $\Wit\bigl(\G(S)\bigr)$ does not contain a pair  of disjoint subsurfaces.
    \item \label{item:main_thm_rel_hyp} $\G(S)$ is hyperbolic relative to thick of order~$0$ subsets if  and only if $\Wit\bigl(\G(S)\bigr)$ contains a pair of disjoint subsurfaces and whenever $Z,W \in \Wit\bigr(\G(S)\bigr)$ are disjoint and co-connected, then $S \sminus Z = W$.
    \item \label{item:main_thm_thick_1} $\G(S)$ is thick of order~$1$ if  there exists a pair $Z,W \in \Wit\bigr(\G(S)\bigr)$ that are disjoint and co-connected so that either $S\sminus (Z \cup W)$ contains a subsurface that is not an annulus with punctures or $Z$ is a punctured disk and $S \sminus (Z \cup W)$ is an annulus with  punctures.
    \item \label{item:main_thm_thick 2} $\G(S)$ is thick of order at most~$2$ if there exists a pair $Z,W \in \Wit\bigr(\G(S)\bigr)$ that are disjoint and co-connected  so that  $S\sminus (Z \cup W)$ is a non-empty disjoint union of punctured annuli. 
\end{enumerate}
\end{theorem}

Item~(\ref{item:main_thm_hyp}) is precisely Theorem \ref{theorem:hyperbolic_graphs_of_multicurves}.
Item~(\ref{item:main_thm_thick_1}) will be proved in Theorems~\ref{theorem:thick case no separating annulus} and~\ref{theorem:thick_punctured_disc_case}, and
Item~(\ref{item:main_thm_thick 2}) is shown in Theorem~\ref{theorem:thick_of_order_2_case}. The sufficient condition in Item~(\ref{item:main_thm_rel_hyp}) will be proved in Theorem~\ref{theorem:relatively_hyperbolic_graphs_of_multicurves}. To see that this condition is necessary, observe that if there exist disjoint and co-connected witnesses $Z,W \in \Wit\bigr(\G(S)\bigr)$ where $S \sminus Z \neq W$, then $Z$ and $W$ must satisfy the conditions for thickness in either Item~(\ref{item:main_thm_thick_1}) or~(\ref{item:main_thm_thick 2}). Since being thick of any order is an obstruction to a space being relatively hyperbolic \cite[Corollary~7.9]{Behrstock_Drutu_Mosher_Thickness}, this implies the necessary direction of Item~(\ref{item:main_thm_rel_hyp}).

 \section{The relatively hyperbolic case}
We now establish a sufficient condition for a non-hyperbolic hierarchical graph of multicurves to be relatively hyperbolic. Our work in subsequent sections will establish that this condition is also necessary. Recall that if $\mf{S}$ is the set of witnesses for a hierarchical graph of multicurves of $S$ and $m$ is a multicurve, then $P_\mf{S}(m) = \{ y \in \ksep_\mf{S}(S) : m\subseteq y \}$. By Corollary \ref{corollary:product regions}, the $P_\mf{S}(\partial Z)$ in Theorem \ref{theorem:relatively_hyperbolic_graphs_of_multicurves} will be thick of order~0.

\begin{theorem}\label{theorem:relatively_hyperbolic_graphs_of_multicurves}
Let $\G(S)$ be a hierarchical graph of multicurves that is not hyperbolic and let $\mf{S} = \Wit\bigl(\G(S)\bigr)$.   Suppose that whenever $Z,W \in \Wit\bigl(\G(S)\bigr)$ are disjoint and co-connected, then $Z = W^c$. Then $\G(S)$ is hyperbolic relative to subsets quasi-isometric to $\{P_\mf{S}(\partial Z) : Z \in \mc{Z}\}$ where  $\mc{Z} = \{Z \in \Wit\bigl(\G(S)\bigr): Z^c \in \Wit\bigl(\G(S)\bigr)\}$.
\end{theorem}

Our proof of Theorem \ref{theorem:relatively_hyperbolic_graphs_of_multicurves} rests on  the following result of the first author, which was originally proved in the context of any hierarchically hyperbolic space. We have translated the result to the special case of a hierarchical graph of multicurves. 
When $W$ and $V$ are two not necessarily connected subsurfaces, $W \subseteq V$ will mean that every component of~$W$ is contained in some component of~$V$.

\begin{theorem}[{Special case of \cite[Theorem 4.3]{russell}}]\label{theorem:special_case_of_isolated_orthogonality}
Let $\mf{S}$ be the set of witnesses for some hierarchical graph of multicurves and let $\mf{T}$ be the set of all (not necessarily connected) subsurfaces of $S$ where each component is an element of $\mf{S}$. Suppose there exists $\mc{I} \subseteq \mf{T}$ such that:
\begin{enumerate}
    \item $S \not \in \mc{I}$.
    \item For all $W,V \in \mf{T}$, if $W$ and $V$ are disjoint, then there exists $I \in \mc{I}$ such that $W,V \subseteq I$.
    \item If $W \in \mf{T}$ and $W \subseteq I_1$, $W \subseteq I_2$ where $I_1,I_2 \in \mc{I}$, then $I_1 = I_2$.
\end{enumerate}
Then, $\ksep_{\mf{S}}(S)$ is hyperbolic relative to the collection $\{P_\mf{S}(\partial I) : I \in \mc{I} \}$. \qedhere
\end{theorem}

The presence of the set $\mf{T}$  in Theorem~\ref{theorem:special_case_of_isolated_orthogonality} is due to the fact that the hierarchically hyperbolic space structure for a hierarchical graph of multicurves includes all the subsurfaces in $\mf{T}$ and not just the witnesses of the graph; see Remark~\ref{remark:disjoint_unions}.

In light of Theorem~\ref{theorem:special_case_of_isolated_orthogonality}, we can prove Theorem~\ref{theorem:relatively_hyperbolic_graphs_of_multicurves} by showing the hypothesis on the set of witnesses implies we can apply Theorem~\ref{theorem:special_case_of_isolated_orthogonality} with  $\mc{I} = \{ Z \sqcup Z^c : Z \in \mc{Z}\}$.

\begin{proof}[Proof of Theorem \ref{theorem:relatively_hyperbolic_graphs_of_multicurves}]
Since $\G(S)$ is not hyperbolic, there exist disjoint witnesses $Z$ and $W$. Moreover, we may arrange $Z$ and $W$ so that $Z^c$ and $W^c$ are connected. This implies the set $\mc{Z} = \{Z \in \Wit\bigl(\G(S)\bigr) : Z^c \in \Wit\bigl(\G(S)\bigr)\}$ is non-empty. 

Let $\mf{T}$ be the  set of all (not necessarily connected) subsurfaces of $S$ where each component is a witness of $\G(S)$. We shall show that the hypotheses of Theorem~\ref{theorem:special_case_of_isolated_orthogonality} are met when $\mc{I} = \{Z \sqcup Z^c : Z \in \mc{Z}\}$.
\begin{enumerate}
    \item By construction $S \not \in \mc{I}$.
    \item Let $W,V \in \mf{T}$ such that $W$ and $V$ are disjoint.
    Recall that $W$ and $V$ may not be connected.
    Let $W'$ be one of the components of $W$.
    Since $V$ is disjoint from $W$, it is in particular disjoint from $W'$, so each component of $V$ is contained in some component of $S \sminus W'$.
    Let $Z$ be a component of $S \sminus W'$ that contains a component of $V$.
    Note that $Z$ and $Z^c$ are both connected subsurfaces.
    Moreover, $Z$ is a witness since it contains a component of $V$, and $Z^c$ is a witness since it contains $W'$, so $Z \sqcup Z^c$ is an element of $\mc{I}$.
    Now every component of $W$ is either contained in $Z$ or contained in $Z^c$, and similarly $V \subseteq S \sminus W' \subseteq Z \sqcup Z^c$.
    Hence, as required, we have $W,V \subseteq Z \sqcup Z^c$.
    
    \item Let $W \in \mf{T}$ and assume there exist $Z_1,Z_2 \in \mc{Z}$ such that $W \subseteq Z_1 \sqcup Z_1^c$ and $W \subseteq Z_2 \sqcup Z_2^c$. Let $W'$ be a component of $W$.  Up to switching $Z_1$ and/or $Z_2$ with its complement, we can assume $W' \subseteq Z_1$ and $W' \subseteq Z_2$. Let $V_i$ be the component of $S \sminus W'$ that contains $Z_i^c$ for $i = 1,2$. If $V_1 \neq V_2$, then $V_1$ and $V_2$ are disjoint witnesses of $\G(S)$ whose complements are both connected. Our assumption on $\G(S)$ then implies $V_1 = V_2^c$ . However,  $V_1$ and $V_2$ are both disjoint from $W'$, which contradicts that $V_1 = V_2^c$ . Therefore, we must have $V_1 = V_2$, which in turn implies  $Z_1^c,Z_2^c \subseteq V_1$. In particular, $V_1^c$ is disjoint from both $Z_1^c$ and $Z_2^c$. Since $V^c_1$, $Z_1^c$, and $Z_2^c$ are all co-connected witnesses of $\G(S)$, this implies $V_1 = Z_1^c$ and $V_1 = Z_2^c$ by hypothesis. Hence, $Z_1 \sqcup Z_1^c = Z_2 \sqcup Z_2^c$.
\end{enumerate}

By Theorem \ref{theorem:special_case_of_isolated_orthogonality}, the above implies    $\ksep_{\mf{S}}(S)$ is hyperbolic relative to $\{P_\mf{S}(\partial (Z \sqcup Z^c) : Z \in \mc{Z} \}$. This implies the conclusion of Theorem~\ref{theorem:relatively_hyperbolic_graphs_of_multicurves} as $\ksep_{\mf{S}}(S)$ is quasi-isometric to~$\G(S)$ by Proposition~\ref{proposition: k qi g} and $\partial (Z \sqcup Z^c) = \partial Z$ for any  subsurface $Z \subseteq S$.
\end{proof}

\section{The thick of order 1 case}\label{section:thick of order 1 case}

We now produce two sufficient conditions, in terms of the set of witnesses, for a hierarchical graph of multicurves to be thick of order~$1$. 
In each case, we show that the conditions imply the existence of a pair of disjoint, co-connected witnesses $Z,W$  so that $\Dis(Z,W)$ contains a large enough connected subgraph to apply Lemma~\ref{lemma:connected_components_are_thick_of_order_1}.

\subsection{Disjoint witness not separated by a punctured annulus.}

Our first condition for a graph of multicurves to be thick of order~$1$ is the existence of a pair of disjoint, co-connected witnesses $Z,W$ so that $Z$ and $W$ are separated by a subsurface that is not a punctured annulus.

\begin{theorem}\label{theorem:thick case no separating annulus}
Let $\G(S)$ be a hierarchical graph of multicurves. Suppose there exist $Z,W \in \Wit\bigl(\G(S)\bigr)$ such that
\begin{itemize}
    \item $Z$ and $W$ are disjoint;
    \item $Z^c$ and $W^c$ are connected;
    \item  $S \sminus (Z \cup W)$ has a component $Y$ that is not an annulus with punctures.
\end{itemize}
Then $\G(S)$ is thick of order 1.
\end{theorem}

When $S$ is a closed surface, Theorems \ref{theorem:hyperbolic_graphs_of_multicurves}, \ref{theorem:relatively_hyperbolic_graphs_of_multicurves}, and \ref{theorem:thick case no separating annulus} are sufficient to classify the hyperbolicity and relative hyperbolicity of $\G(S)$. In this case, the classification will be determined by the Euler characteristic, $\chi(\cdot)$, of the co-connected witnesses. For ease of stating the theorem, we restrict to the case of $g \ge 2$ as for $g=0$ there are no graphs of multicurves and for $g=1$ all hierarchical graphs of multicurves are quasi-isometric to the curve graph  by Proposition \ref{proposition: k qi g}.

\begin{corollary}
Let $\G(S)$ be a hierarchical graph of multicurves on a closed surface $S = S_{g,0}$ with $g \geq 2$. Let $\chi_{\min}$ be the least negative Euler characteristic of a co-connected element of $\mf{S} =\Wit\bigl(\G(S)\bigr)$.
\begin{itemize}
    \item If $|\chi_{\min}| >  \frac{1}{2}|\chi(S)|$, then $\G(S)$ is hyperbolic.
    \item If $|\chi_{\min}| = \frac{1}{2} |\chi(S)|$, then $\G(S)$ is relatively hyperbolic, and each peripheral subset is quasi-isometric to $P_\mf{S}(\partial Z)$ for some $Z \in \mf{S}$ with $Z$ co-connected and $\chi(Z) = \chi_{\min}$.
    \item If $|\chi_{\min}| < \frac{1}{2} |\chi(S)|$, then $\G(S)$ is thick of order~$1$.
\end{itemize}
\end{corollary}

\begin{proof}

For any subsurface $Z \subseteq S$, recall that $\chi(S) = \chi(Z) + \chi(Z^c)$.
 Moreover, an essential subsurface will always have non\hyp{}positive Euler characteristic, so $\vert\chi(S)\vert = \vert\chi(Z)\vert + \vert\chi(Z^c)\vert$.
Thus, if $|\chi_{\min}| > \frac{1}{2} |\chi(S)|$, then for every  witness $W$ of $\G(S)$, no component of $S \sminus W$ can be a witness of $\G(S)$. Therefore $\G(S)$ is hyperbolic by Theorem~\ref{theorem:hyperbolic_graphs_of_multicurves}.

Suppose $|\chi_{\min}| = \frac{1}{2} |\chi(S)|$. If $W \in \Wit\bigl( \G(S) \bigr)$ so that $W$ is co-connected and $\chi(W) = \chi_{\min}$, then $\chi(W^c) = \chi(W) = \frac{1}{2}\chi(S)$. Since $S$ is a closed surface, this implies $W^c$ is in the $\mcg(S)$-orbit of $W$. In particular,  $W^c \in  \Wit\bigl( \G(S) \bigr)$, implying $\G(S)$ is not hyperbolic by Theorem~\ref{theorem:hyperbolic_graphs_of_multicurves}. 
By Theorem~\ref{theorem:relatively_hyperbolic_graphs_of_multicurves}, $\G(S)$ will be relatively hyperbolic with the specified peripherals if we can show for any pair of disjoint, co-connected witnesses $W$ and $Z$ for $\G(S)$ we have $W^c = Z$ and $\chi(W) = \chi_{\min}$. Let $W$ and $Z$ be a pair of disjoint, co-connected witnesses for~$\G(S)$.

Because $Z \subseteq W^c$, we know $W^c$ is also a co-connected witness of $\G(S)$. 
Since $\vert \chi_{\min} \vert =\frac{1}{2} \vert \chi(S) \vert$ and $W$ and $W^c$ are both witnesses, we must have that $\chi(W) = \chi(W^c) = \chi_{\min}$.
This implies $Z = W^c$,  because $\chi(W^c) = \chi_{\min}$ prevents any proper subsurface of $W^c$ from being a witness.

Finally, assume $|\chi_{\min}| < \frac{1}{2} | \chi(S)|$. Let $W \in \Wit\bigl( \G(S) \bigr)$ so that $W$ is co-connected and $\chi(W) = \chi_{\min}$. As in the previous case, $S$ being closed and $W$ being co-connected implies there exists a subsurface $Z \subseteq W^c$ so that  $Z$ is in the $\mcg(S)$-orbit of $W$. In particular, $Z$ is a co-connected witness of $\G(S)$ that is disjoint from $W$ and also has $|\chi(Z)| < \frac{1}{2} |\chi(S)|$. Since $|\chi(W)| + |\chi(Z)| < |\chi(S)|$ and $S$ is closed, there must be a component of $S \sminus (W \cup Z)$ that is not an annulus with punctures. Therefore, $W$ and $Z$ satisfy the hypotheses of Theorem \ref{theorem:thick case no separating annulus} and $\G(S)$ is thick of order~1.
\end{proof}

\begin{proof}[Proof of Theorem~\ref{theorem:thick case no separating annulus}]
By the hypotheses of Theorem~\ref{theorem:thick case no separating annulus}, there exists a pair of disjoint, co-connected witnesses $Z_0, W_0$, so that $S \sminus (Z_0 \cup W_0)$ has a component $Y_0$ that is not an annulus or an annulus with punctures. Thus, there exists a subsurface $Y \subseteq Y_0$ so that $Y$ is homeomorphic to $S_{0,0}^3$. Moreover, after possibly switching $W_0$ and $Z_0$, we can choose $Y$ to meet $W_0$ in exactly one curve and so that $S \sminus (W_0 \cup Y)$ is connected. Let $Z = S \sminus (W_0 \cup Y)$ and $W = W_0$; see Figure~\ref{figure:separated_by_pants} for an example of $Z$, $W$, and $Y$. Because $Z_0 \subseteq Z$ and $W_0 \subseteq W$, we have that  $Z$ and $W$ are disjoint elements of $\Wit\bigl( \mc{G}(S)\bigr)$ with $Z^c$ and $W^c$ connected. By our choice of $Y$ and $Z$, the intersection $\partial Z \cap \partial Y$ contains two curves. Note, this implies that the genus of~$S$ must be at least~$1$.

\begin{figure}[ht]
    \centering
    \def\svgscale{.8}
    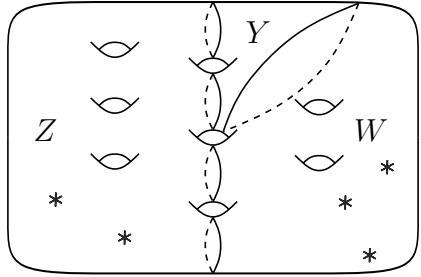
    \caption{The subsurfaces $Z$ (left), $W$ (right), and $Y$ (center).}
    \label{figure:separated_by_pants}
\end{figure}

We will use Lemma~\ref{lemma:connected_components_are_thick_of_order_1} to prove $\ksep_\mf{S}(S)$ is thick of order~$1$ by finding a connected subset of $\Dis(Z,W)$ whose corresponding product regions coarsely cover all of $\ksep_\mf{S}(S)$. Since the order of thickness is a quasi-isometry invariant, this will show $\G(S)$ is also thick of order~$1$.

Let $\PDis(Z,W)$ be the subgraph of $\Dis(Z,W)$ spanned by the  set of vertices $m \in \Dis(Z,W)$ such that there exists  $f \in \pmcg(S)$ (instead of in $\mcg(S)$) so that $S \sminus m$ has one component containing $ f(Z)$ and another component containing $f(W)$. By Lemma~\ref{lemma:connected_components_are_thick_of_order_1}, if $\PDis(Z,W)$ is connected, then $\bigcup_{m \in \PDis(Z,W)} P_\mf{S}(m)$ is thick of order at most $1$. Since there are only finitely many $\pmcg(S)$-orbits of vertices of $\ksep_\mf{S}(S)$, there exists $D\geq 0$ depending only on $S$  so that every vertex of $\ksep_\mf{S}(S)$ is contained in the $D$-neighborhood of some product region $P_\mf{S}(m)$ for $m \in \PDis(Z,W)$. Therefore, $\ksep_\mf{S}(S)$  will be thick of order at most~$1$ if $\bigcup_{m \in \PDis(Z,W)} P_\mf{S}(m)$ is a thick of order at most~$1$ subset.

We will use Lemma~\ref{lemma:putman_trick} to prove the connectedness of $\PDis(Z, W)$, using the pure mapping class group of~$S$ as the group acting on $\PDis(Z, W)$, and $\partial Z$ as the base vertex.
For every vertex $m$ of $\PDis(Z, W)$  there exists $f \in \pmcg(S)$ so that a component of $S \sminus m$ contains  $ f(Z)$.
Both $\partial f(Z)$ and $m$ are vertices of $\PDis(Z, W)$, and moreover, both $m$ and  $\partial f(Z)$ are connected to  $m \cup \partial f(Z)$  by successively adding curves.
Hence $m$ is connected to some $\pmcg(S)$ translate of $\partial Z$.
It therefore suffices to find a generating set $X$ for $\pmcg(S)$ such that $Z$ is connected to $\phi(Z)$ for each $\phi\in X \sqcup X^{-1}$.

 Arrange a set of standard generating curves for $\pmcg(S)$ (see Definition~\ref{definition:generating curves}) in such a way that every generating curve either is disjoint from $\partial Z$ or intersects exactly two curves of $\partial Z$ exactly once each; see Figure~\ref{figure:gen_twists_separated_by_pants_Humphries} for examples.
Recall that  the genus of $S$ will always be at least 1 when the hypotheses of Theorem~\ref{theorem:thick case no separating annulus} are satisfied.

\begin{figure}[ht]
\begin{tabular}{cc}
\def\svgscale{.7}
        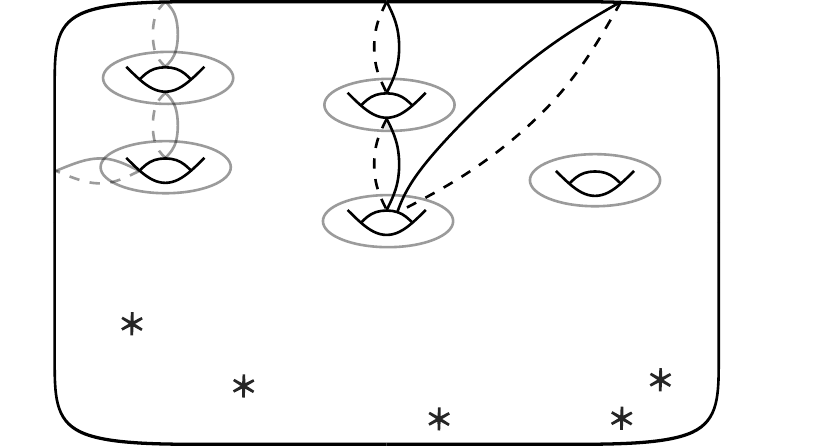 &  
             \def\svgscale{.2}
\begingroup%
  \makeatletter%
  \providecommand\color[2][]{%
    \errmessage{(Inkscape) Color is used for the text in Inkscape, but the package 'color.sty' is not loaded}%
    \renewcommand\color[2][]{}%
  }%
  \providecommand\transparent[1]{%
    \errmessage{(Inkscape) Transparency is used (non-zero) for the text in Inkscape, but the package 'transparent.sty' is not loaded}%
    \renewcommand\transparent[1]{}%
  }%
  \providecommand\rotatebox[2]{#2}%
  \newcommand*\fsize{\dimexpr\f@size pt\relax}%
  \newcommand*\lineheight[1]{\fontsize{\fsize}{#1\fsize}\selectfont}%
  \ifx\svgwidth\undefined%
    \setlength{\unitlength}{748.09260652bp}%
    \ifx\svgscale\undefined%
      \relax%
    \else%
      \setlength{\unitlength}{\unitlength * \real{\svgscale}}%
    \fi%
  \else%
    \setlength{\unitlength}{\svgwidth}%
  \fi%
  \global\let\svgwidth\undefined%
  \global\let\svgscale\undefined%
  \makeatother%
  \begin{picture}(1,0.4824959)%
    \lineheight{1}%
    \setlength\tabcolsep{0pt}%
    \put(0,0){\includegraphics[width=\unitlength,page=1]{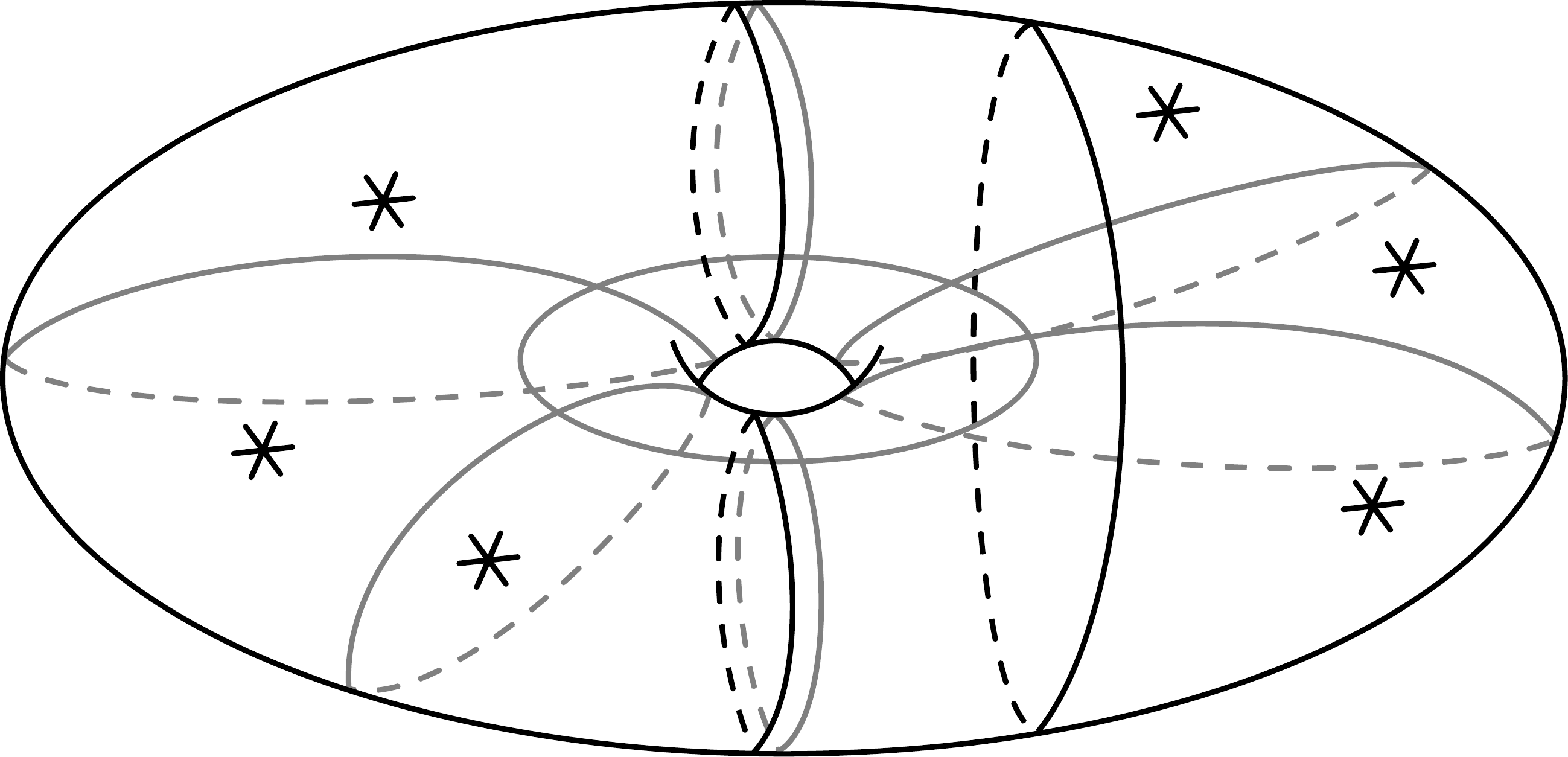}}%
    \put(0.02097952,0.43417265){\color[rgb]{0,0,0}\makebox(0,0)[lt]{\begin{minipage}{0.05266188\unitlength}\raggedright $Z$\end{minipage}}}%
    \put(0.54861088,0.40417756){\color[rgb]{0,0,0}\makebox(0,0)[lt]{\lineheight{1.25}\smash{\begin{tabular}[t]{l}$Y$\end{tabular}}}}%
    \put(0.94051623,0.05946326){\color[rgb]{0,0,0}\makebox(0,0)[lt]{\lineheight{1.25}\smash{\begin{tabular}[t]{l}$W$\end{tabular}}}}%
  \end{picture}%
\endgroup%
 
\end{tabular}
\caption{Dehn twists around the gray curves give a  generating set for $\pmcg(S)$.} 
    \label{figure:gen_twists_separated_by_pants_Humphries}
\end{figure}

Let $X$ be the set of (left) Dehn twists about this set of standard generating curves.
Let $\phi \in X \sqcup X^{-1}$ (that is, either a left or a right Dehn twist about one of the generating curves) and let $a$ be the core curve of the twist~$\phi$.
If $a$ is disjoint from $\partial Z$, then $\phi(\partial Z) = \partial Z$ and so $\partial Z$ is trivially connected to $\phi(\partial Z)$.

Suppose now that $a$ intersects $\partial Z$.
Then $a$ intersects exactly two curves of $\partial Z$, $c_1$ and $c_2$, exactly once each.
The twist about $a$ happens in a subsurface that is a twice\hyp{}holed torus, as shown in Figure~\ref{fig:twist in twice holed torus}. 

\begin{figure}[ht]
     \centering
 \begin{subfigure}[b]{0.45\textwidth}
\def\svgscale{.6}
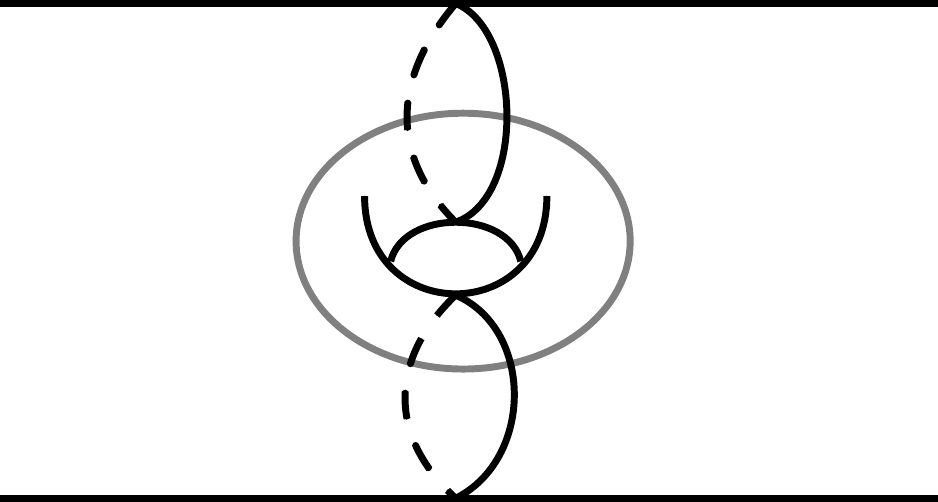
\subcaption{A generating curve $a$ intersecting two curves, $c_1$ and $c_2$ , of $\partial Z$.}
\end{subfigure}
\qquad
\begin{subfigure}[b]{0.45\textwidth}
\def\svgscale{.6}
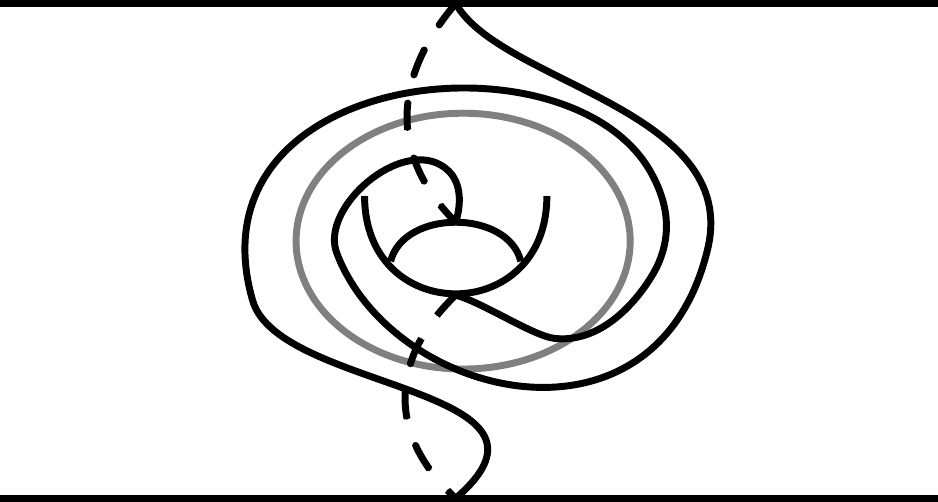
\subcaption{The result of applying a twist $\phi$ about $a$ to $c_1$ and $c_2$.}
\end{subfigure}
\caption{}
\label{fig:twist in twice holed torus}
 \end{figure}

The union of $Z$ and $\phi(Z)$ is contained in a subsurface $\hat{Z}$ that is obtained by adding a pair of pants to $Z$ by gluing one boundary component to $c_1$ and another to $c_2$.
The complement of $\hat{Z}$ is homeomorphic to $W$, because $\hat{Z}$ is homeomorphic to $Z \cup Y = S \sminus W$.
Moreover, $S \sminus \hat{Z}$ meets $Z$ along all but two of the curves of $\partial Z$, just as $W$ does.
Hence there exists $f \in \pmcg(S)$ such that $f(Z) = Z$ and $f(W)=S \sminus \hat{Z}$, and hence $f(W)$ is disjoint from both $Z$ and $\phi(Z)$; see Figure~\ref{figure:avoiding_twist} for an example.
The multicurves  $\partial f(W)$,  $\partial Z \cup \partial f(W)$, and $\phi(\partial Z) \cup \partial f(W)$ are all vertices of $\PDis(Z, W)$.
Hence we have a path in $\PDis(Z, W)$, by connecting consecutive vertices in the following sequence by adding and removing curves:
\[\partial Z,\; \partial Z \cup \partial f(W),\; \partial f(W),\; \partial f(W) \cup \phi(\partial Z),\; \phi(\partial Z). \qedhere \]

\begin{figure}[ht]
    \centering
    \def\svgscale{.8}
    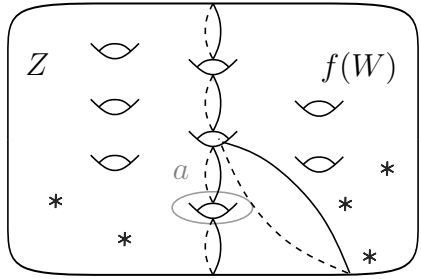
    \caption{The subsurface $f(W)$ is in the $\pmcg(S)$-orbit of $W$ and disjoint from both the standard generating curve $a$ and the subsurface $Z = f(Z)$.}
    \label{figure:avoiding_twist}
\end{figure}

\end{proof}

\subsection{Disjoint witnesses separated by a punctured annulus.}

Our second condition for a hierarchical graph of multicurves to be thick of order~$1$ requires the existence of a co-connected witness $Z$ that is a punctured disk and co-connected witness $W$ that is separated from $Z$ by a punctured annulus.

\begin{theorem} \label{theorem:thick_punctured_disc_case}
Let $\G(S)$ be a hierarchical graph of multicurves. Suppose there exist $Z,W \in \Wit\bigl(\G(S)\bigr)$ such that
\begin{itemize}
    \item $Z$ and $W$ are disjoint;
    \item $Z^c$ and $W^c$ are connected;
    \item $Z$ is a punctured disk;
    \item $S \sminus (Z \cup W)$ is an annulus with at least one puncture.
\end{itemize}
Then $\G(S)$ is thick of order 1.
\end{theorem}

Note that in particular, this combines with Theorem~\ref{theorem:hyperbolic_graphs_of_multicurves} and Theorem~\ref{theorem:relatively_hyperbolic_graphs_of_multicurves} to give a simple classification for the case when $S$ is a punctured sphere.

\begin{corollary}\label{corollary:punctured_sphere_case}
Let $\G(S)$ be a hierarchical graph of multicurves on a punctured sphere $S = S_{0,p}$. Let $r$ be the minimal number of punctures of $S$ contained in a co-connected witness for $\G(S)$.
\begin{itemize}
    \item If $r > p/2$, then $\G(S)$ is hyperbolic.
    \item If $r = p/2$, then $\G(S)$ is relatively hyperbolic, and each peripheral subset is quasi-isometric to $P_\mf{S}(c)$ for some curve $c$ separating $p/2$ punctures  of $S$ on each side.
    
    \item If $r < p/2$, then $\G(S)$ is thick of order~$1$.
\end{itemize}
\end{corollary}
\begin{proof}
For a subsurface $Y$ of $S$, let $p(Y)$ be the number of punctures of $S$ contained in~$Y$.

Assume $r > \frac{p}{2}$ and suppose there exist  $Z_1, Z_2 \in \Wit\bigl( \G(S) \bigr)$ with $Z_1$ disjoint from~$Z_2$. Let $W$ be the component of $Z_1^c$ that contains $Z_2$. Then, $W$ and $W^c$ are both witnesses of $\G(S)$ as they contain $Z_2$ and $Z_1$ respectively. Since both $W$ and $W^c$ are connected, we must have $p(W) > p/2$ and $p(W^c) > p/2$. But since this is impossible, it must be the case that no pair of elements of $\Wit\bigl( \G(S) \bigr)$ are disjoint. Thus, $\G(S)$ is hyperbolic by  Theorem \ref{theorem:hyperbolic_graphs_of_multicurves}.
 
Assume $r = \frac{p}{2}$ and suppose $Z,W \in \Wit\bigl(\G(S)\bigr)$ are disjoint and co-connected. This implies that  $p(Z) \geq p/2$ and $p(W) \geq p/2$. However, since $W \subseteq Z^c$ implies $p(W) \leq p - p(Z)$, we must have $p(W) = p(Z) = p/2$. Therefore $Z =W^c$ and the conclusion follows by Theorem~\ref{theorem:relatively_hyperbolic_graphs_of_multicurves} because $\partial Z$ is a curve separating $p/2$ punctures on each side.
 
 Finally, assume $r < \frac{p}{2}$ and let $Z$ be a co-connected witness of $\G(S)$ that is homeomorphic to $D_r$, a disk with $r$ punctures. Since $r < \frac{p}{2}$, $S \sminus Z$ contains a subsurface $W$ so that $S \sminus W$ is connected and $W$ is a disk with $r$ punctures. In particular, $W$ is a co-connected witness of $\G(S)$ that is disjoint from $Z$ and $S \sminus (Z \cup W)$ is an annulus with $p-2r \geq 1$ punctures. Thus, $\G(S)$ is thick of order~$1$ by Theorem \ref{theorem:thick_punctured_disc_case}.
\end{proof}

\begin{proof}[Proof of Theorem~\ref{theorem:thick_punctured_disc_case}]
Notice that if the number of punctures of $S$ is $p$, then for each $k \leq p$ there is exactly one $\mcg(S)$\hyp{}orbit of subsurfaces homeomorphic to $D_k$.
Similarly, there is exactly one $\mcg$\hyp{}orbit of subsurfaces homeomorphic to $S \sminus D_k$.
Let $K$ be the maximal number so that the complement of a disk with $K$ punctures is a witness for $\G(S)$.
We know from the hypotheses of Theorem \ref{theorem:thick_punctured_disc_case} that there exists a punctured disk $D$ that is a witness for $\G(S)$ and is separated from another witness by a punctured annulus $A$.
Since $D \cup A$ is a punctured disk whose complement is a witness, it has at most $K$ punctures. Hence $D$ has at most $K-1$ punctures, and we can conclude that any subsurface homeomorphic to $D_{K-1}$ or $D_K$ is a witness for $\G(S)$ and that $K \geq 4$ (since a copy of $D_2$  can never be a witness).

By Lemma~\ref{lemma:connected_components_are_thick_of_order_1}, it is sufficient to find a pair of disjoint,  co-connected witnesses $Z$ and $W$ so that $\Dis(Z, W)$ is connected. Let $Z'$ be a subsurface of $S$ that is homeomorphic to $D_K$ and let $Z$ be a subsurface of $Z'$ that is homeomorphic to~$D_{K-1}$. Let $W = S \sminus Z'$. As noted above, $Z'$, $Z$, and $W$ are all witnesses of $\G(S)$.
Note that any $\mcg(S)$\hyp{}translate of $\partial Z$ or $\partial W$ is a vertex of $\Dis(Z, W)$.

We will use Lemma~\ref{lemma:putman_trick} to prove that $\Dis(Z, W)$ is connected.
The group we use is the mapping class group of~$S$ with generating set depending on the genus of~$S$.
Specifically, if $S$ has genus at least~2, then we will use a generating set as in Theorem~\ref{theorem:humphries_generators}, if $S$ has genus~1, then we will use a generating set as in Theorem~\ref{theorem:gervais_generators}, and if $S$ has genus~0, then we will use a generating set as in Theorem~\ref{theorem:braid_generators}.

We take the base vertex for Lemma~\ref{lemma:putman_trick} to be~$\partial W$.
If $S$ has genus, then our generating set for $\mcg(S)$ includes Dehn twists along certain curves.
We choose these generating curves so that the intersection with $\partial W$ is minimal as in Figure~\ref{fig:one_punctured_disc_with_generators}.
 In particular, in the case of genus at least~2, only the curves $e_i$ from Theorem~\ref{theorem:humphries_generators} will intersect~$\partial W$, and they will each do so at most twice.
Similarly, we choose the twice punctured disks in which the half twist generators take place so that their boundaries intersect $\partial W$ minimally.
In particular, if the two punctures are in the same component of $S \sminus \partial W$ then this intersection is empty, and if they are not, then the boundary of the disk intersects~$\partial W$ twice (Figure~\ref{fig:one_punctured_disc_half_twist_generator}).
Let $X$ be the generating set for $\mcg(S)$ we obtain after arranging these minimal intersections.

\begin{figure}[ht]
    \begin{subfigure}[b]{0.45\textwidth}
    \centering
    \def\svgscale{.5}
    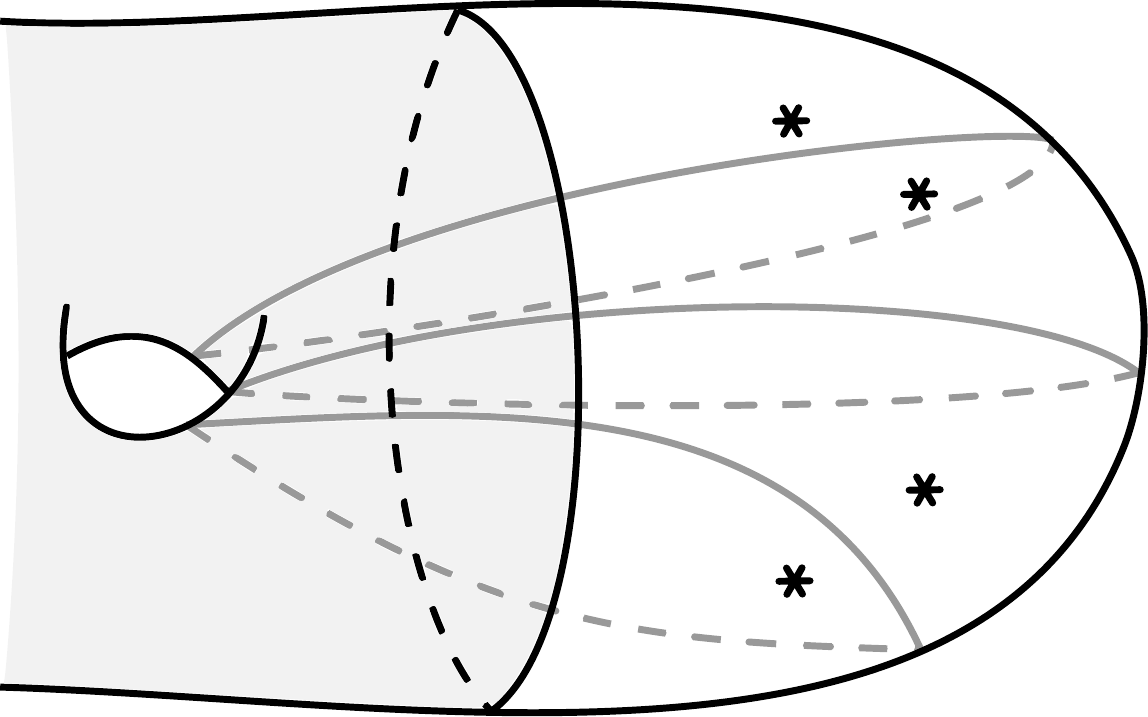
    \caption{In the case where $S$ has positive genus: the base vertex $\partial W$ with the standard generating curves that intersect it (in gray). The subsurface $W$ is shaded.}
    \label{fig:one_punctured_disc_with_generators}
    \end{subfigure}
    \quad
    \begin{subfigure}[b]{0.45\textwidth}
    \centering
    \def\svgscale{.5}
    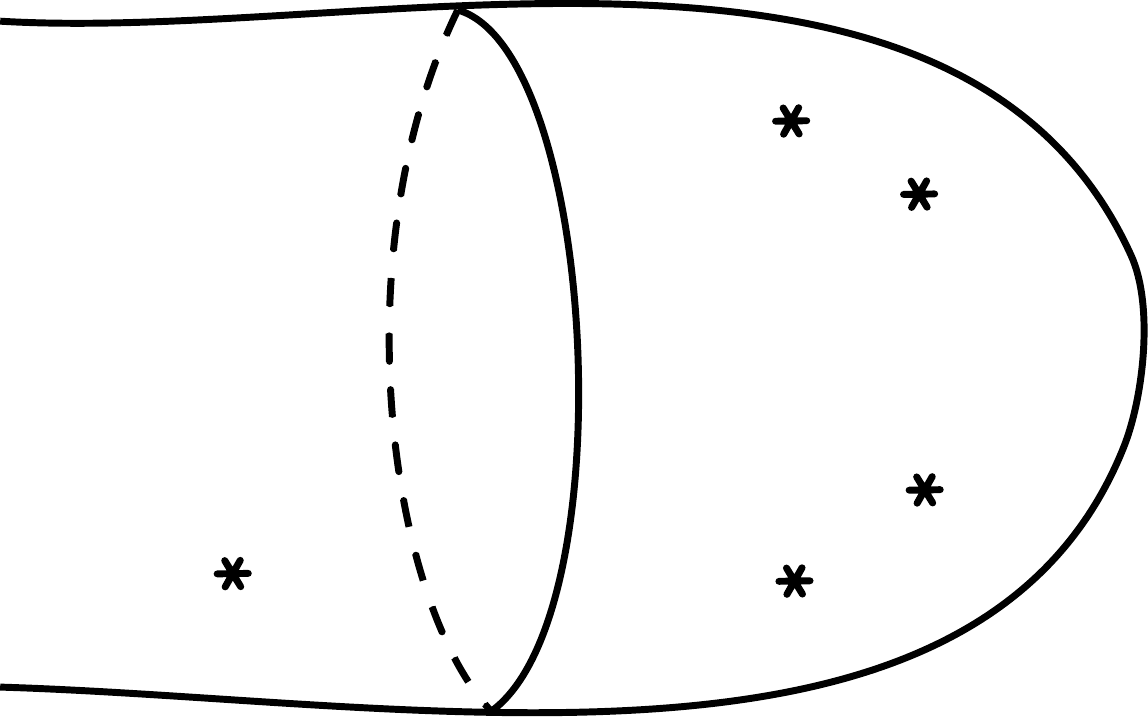
    \caption{We choose half twist generators so that the boundary of the twice punctured disk in which they take place intersects $\partial W$ minimally.}
    \label{fig:one_punctured_disc_half_twist_generator}
    \end{subfigure}
    \caption{}
\end{figure}

For the first hypothesis of Lemma~\ref{lemma:putman_trick}, let $m$ be a vertex of $\Dis(Z,W)$. By definition,   $S \sminus m$ contains $f(W)$ for some $f \in \mcg(S)$.
Then $m \cup \partial f(W)$  and $\partial f(W)$ are both vertices of $\Dis(Z, W)$. Moreover, both $m$ and $\partial f(W)$ are connected to $m \cup \partial f(W)$ in $\Dis(Z,W)$ by successively adding curves. Hence any vertex $m$ of $\Dis(Z, W)$ is connected to a vertex in the $\mcg(S)$-orbit of $\partial W$.

For the second hypothesis of Lemma~\ref{lemma:putman_trick}, we will show $\partial W$ is connected to $\phi(\partial W)$ for all $\phi \in X \sqcup X^{-1}$.  Assume first that $\phi$ is a (forward or backward) half twist swapping two punctures of~$S$, and let $Y$ be the twice punctured disk in which the half twist takes place.
Recall that we chose $Y$ so that $\partial Y$ intersects $\partial W$ minimally.
If both of the punctures are contained in $W$ or in $S \sminus W$ then $\phi$ will not affect the base vertex~$\partial W$.
Suppose that $\phi$ swaps punctures on either side of~$\partial W$; here the boundary of $Y$ will intersect $\partial W$ twice.
Then $\phi(\partial W)$ intersects  $\partial W$ twice and forms two once punctured bigons, one on each side; see Figure~\ref{fig:one_punctured_disc_half_twist} for an example.
Since there are $K-1$ punctures of $S$ that are not contained in either $W$ or $Y$, there exists a curve $a$ disjoint from $\partial W$ and $\phi(\partial W)$ so that $a$ cuts off a $(K-1)$\hyp{}times punctured disk.
Since $a$ is a vertex of $\Dis(Z, W)$, the sequence \[\partial W, \partial W \cup a, a , a \cup \phi(\partial W),  \phi(\partial W) \]   is a path from $\partial W$ to $\phi(\partial W)$ in $\Dis(Z,W)$.

\begin{figure}[ht]
    \centering
    \def\svgscale{.6}
    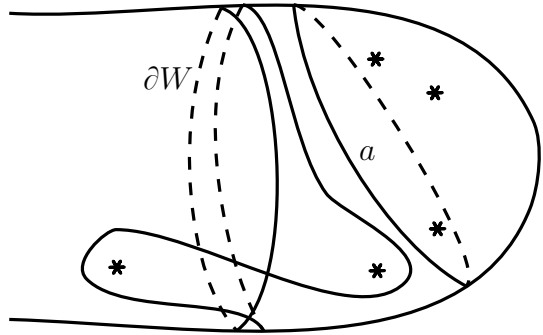
    \caption{A half twist affecting $\partial W$.}
    \label{fig:one_punctured_disc_half_twist}
\end{figure}

Now assume that $\phi$ is a (left or right) Dehn twist about a  standard generating curve. 
If $\phi$ is a twist around a standard generating curve that does not intersect the base vertex~$\partial W$, then $\phi(\partial W)$ is equal to $\partial W$. Thus, we assume $\phi$ is a twist around a standard generating curve that intersects~$\partial W$.
Figure~\ref{fig:one_punctured_disc_with_generators} shows~$\partial W$ along with the  standard generating curves that intersect it.
An example of a twist along one of these curves is shown in Figure~\ref{fig:one_punctured_disc_twisted}.
We will show that there is a path in $\Dis(Z, W)$ from $\partial W$ to~$\phi(\partial W)$.
In fact, we will show that $\partial W$ is connected in $\Dis(Z, W)$ to any vertex $c$ of $\Dis(Z, W)$ such that:
\begin{itemize}
    \item $c$ is in the pure mapping class group orbit of $\partial W$ (in particular, the punctured disks cut off by $c$ and by $\partial W$ contain the same set of punctures);
    \item the punctured disks cut off by $c$ and by $\partial W$ intersect exactly in two punctured bigons (one of which may contain no punctures).
\end{itemize}
We will call the set of curves satisfying these conditions $\Omega(\partial W)$.
The curve $\phi(\partial W)$ is in $\Omega(\partial W)$.

\begin{figure}[ht]
    \centering
    \def\svgscale{.6}
    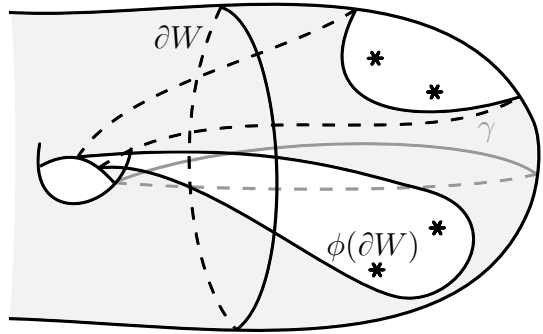
    \caption{The result $\phi(\partial W)$ of twisting $\partial W$ around a generating curve $\gamma$ (with $\phi(W)$ shaded).}
    \label{fig:one_punctured_disc_twisted}
\end{figure}

We will prove $\partial W$ can be connected to every curve in $\Omega(\partial W)$ by inducting on the following quantity $M(c)$: 
for $c \in \Omega(\partial W)$, let $B, B'$ be the two punctured bigons in the intersection of the punctured discs cut off by $\partial W$ and by~$c$. Define $M(c)$ to be the minimum of the number of punctures in $B$ and the number of punctures in $B'$.

For the base case of the induction, suppose that a curve $c$ is in $\Omega(\partial W)$ with $M(c)=0$.
Then the punctured disks cut off by $c$ and $\partial W$ intersect in a  bigon with $K$~punctures and a  bigon with $0$~punctures. The only way for two subsurfaces of $S$ homeomorphic to $D_K$ to intersect in a subsurface that is also homeomorphic to $D_K$ is for the subsurfaces to be isotopic. Thus their boundary curves, $c$ and $\partial W$, must be isotopic.

Now let $c$ be any curve in $\Omega(\partial W)$ with $M(c)>0$.
We will find a curve $c' \in \Omega(\partial W)$ so that $c$ and $c'$ are connected in $\Dis(Z, W)$ and so that $M(c')<M(c)$.

Let $C$ be the punctured disk cut off by $c$.
To make it easier to illustrate the steps, we will work in the subsurface given by the union of the punctured disks $C$ and $W^c$.
This is a planar subsurface with $K$ punctures and two boundary components; see Figure~\ref{fig:punctured_circle_a1}.
Let $B$ and $B'$ denote the two punctured bigons in the intersection of $C$ and $W^c$, with $B$ being the punctured bigon that contains fewer punctures.
Let $R$ be the rectangle in $C^c$ with two edges in $\partial W$ and two edges in $c$.

\begin{figure}[ht]
    \centering
    \def\svgscale{.4}
    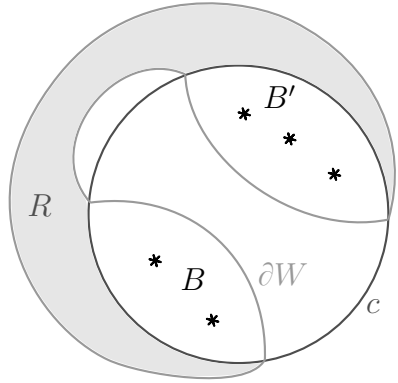
    \caption{The union of $W^c$ and the punctured disk $C$ cut off by a curve $c$ in $\Omega(\partial W)$ is a planar subsurface with $K$ punctures and two boundary components.}
    \label{fig:punctured_circle_a1}
\end{figure}

Choose a puncture $\rho$ in $B$ and let $e$ be a properly embedded arc in $C$ that is contained in $B$ and that separates $\rho$ from the other punctures of $C$; see Figure~\ref{fig:punctured_circle_b1}.
Let $d$ be the curve in $C$ which does not intersect~$e$, and which cuts off a copy $D$ of $D_{K-1}$ containing all punctures of $C$ except for~$\rho$.
The curve $d$ is in the mapping class group orbit of $\partial Z$ and hence is a vertex of $\Dis(Z, W)$.

\begin{figure}[ht]
    \centering
    \begin{minipage}{0.49\textwidth}
        \centering
            \def\svgscale{.4}
\begingroup%
  \makeatletter%
  \providecommand\color[2][]{%
    \errmessage{(Inkscape) Color is used for the text in Inkscape, but the package 'color.sty' is not loaded}%
    \renewcommand\color[2][]{}%
  }%
  \providecommand\transparent[1]{%
    \errmessage{(Inkscape) Transparency is used (non-zero) for the text in Inkscape, but the package 'transparent.sty' is not loaded}%
    \renewcommand\transparent[1]{}%
  }%
  \providecommand\rotatebox[2]{#2}%
  \newcommand*\fsize{\dimexpr\f@size pt\relax}%
  \newcommand*\lineheight[1]{\fontsize{\fsize}{#1\fsize}\selectfont}%
  \ifx\svgwidth\undefined%
    \setlength{\unitlength}{360.62429115bp}%
    \ifx\svgscale\undefined%
      \relax%
    \else%
      \setlength{\unitlength}{\unitlength * \real{\svgscale}}%
    \fi%
  \else%
    \setlength{\unitlength}{\svgwidth}%
  \fi%
  \global\let\svgwidth\undefined%
  \global\let\svgscale\undefined%
  \makeatother%
  \begin{picture}(1,0.98084565)%
    \lineheight{1}%
    \setlength\tabcolsep{0pt}%
    \put(0,0){\includegraphics[width=\unitlength,page=1]{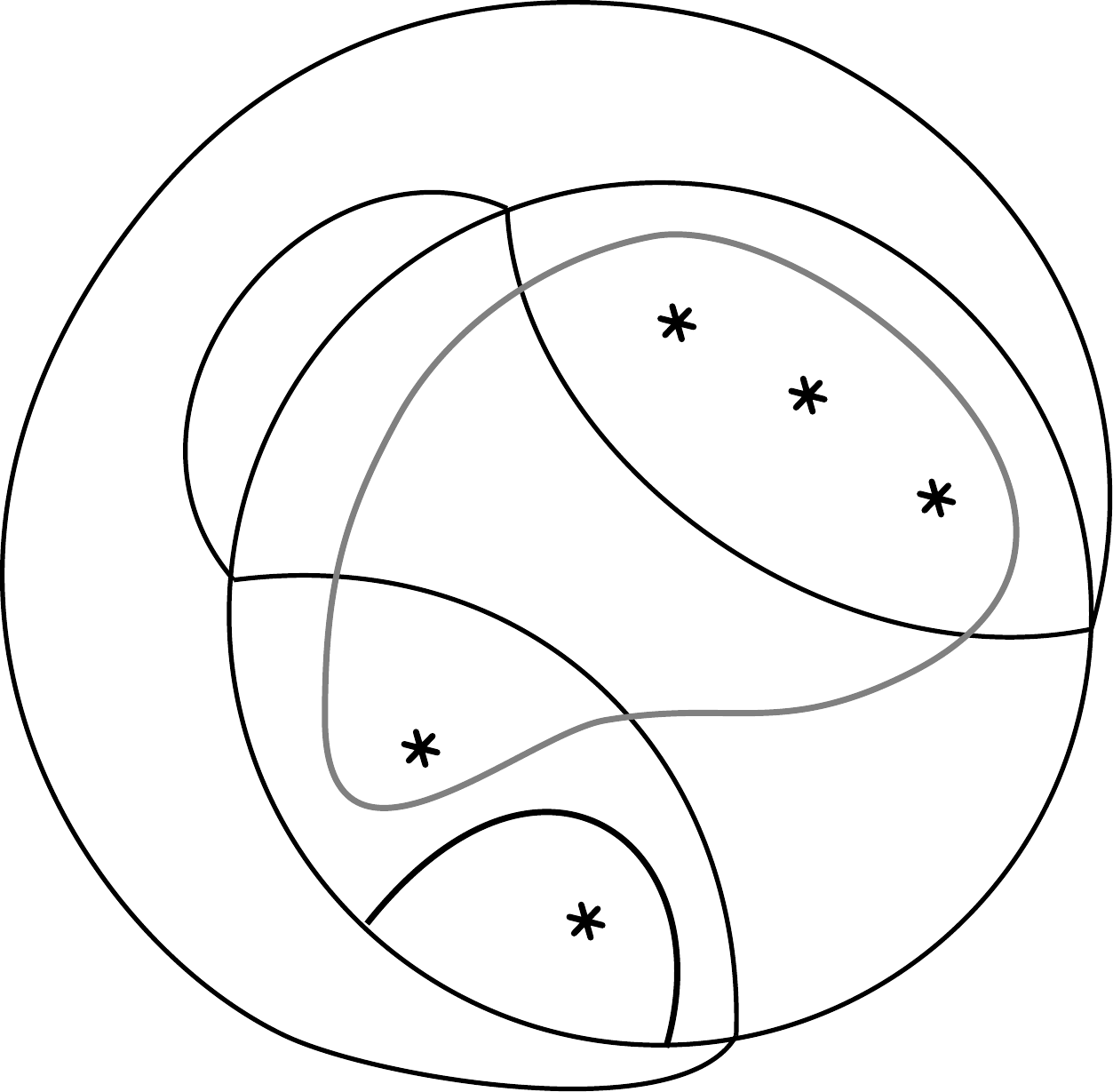}}%
    \put(0.72707438,0.28315113){\color[rgb]{0.50196078,0.50196078,0.50196078}\makebox(0,0)[lt]{\lineheight{1.25}\smash{\begin{tabular}[t]{l}$d$\end{tabular}}}}%
    \put(0,0){\includegraphics[width=\unitlength,page=2]{punctured_circle_b2.pdf}}%
    \put(0.42867038,0.50365504){\color[rgb]{0.50196078,0.50196078,0.50196078}\makebox(0,0)[lt]{\lineheight{1.25}\smash{\begin{tabular}[t]{l}$D$\end{tabular}}}}%
    \put(0.37606495,0.16130234){\color[rgb]{0,0,0}\makebox(0,0)[lt]{\lineheight{1.25}\smash{\begin{tabular}[t]{l}$e$\end{tabular}}}}%
    \put(0.51480495,0.10247886){\color[rgb]{0,0,0}\makebox(0,0)[lt]{\lineheight{1.25}\smash{\begin{tabular}[t]{l}$\rho$\end{tabular}}}}%
  \end{picture}%
\endgroup%

            \caption{The arc~$e$, and the curve $d$ disjoint from $c$.}
           \label{fig:punctured_circle_b1}
    \end{minipage}
    \begin{minipage}{0.49\textwidth}
        \centering
            \def\svgscale{.4}
\begingroup%
  \makeatletter%
  \providecommand\color[2][]{%
    \errmessage{(Inkscape) Color is used for the text in Inkscape, but the package 'color.sty' is not loaded}%
    \renewcommand\color[2][]{}%
  }%
  \providecommand\transparent[1]{%
    \errmessage{(Inkscape) Transparency is used (non-zero) for the text in Inkscape, but the package 'transparent.sty' is not loaded}%
    \renewcommand\transparent[1]{}%
  }%
  \providecommand\rotatebox[2]{#2}%
  \newcommand*\fsize{\dimexpr\f@size pt\relax}%
  \newcommand*\lineheight[1]{\fontsize{\fsize}{#1\fsize}\selectfont}%
  \ifx\svgwidth\undefined%
    \setlength{\unitlength}{360.62429115bp}%
    \ifx\svgscale\undefined%
      \relax%
    \else%
      \setlength{\unitlength}{\unitlength * \real{\svgscale}}%
    \fi%
  \else%
    \setlength{\unitlength}{\svgwidth}%
  \fi%
  \global\let\svgwidth\undefined%
  \global\let\svgscale\undefined%
  \makeatother%
  \begin{picture}(1,0.98084565)%
    \lineheight{1}%
    \setlength\tabcolsep{0pt}%
    \put(0,0){\includegraphics[width=\unitlength,page=1]{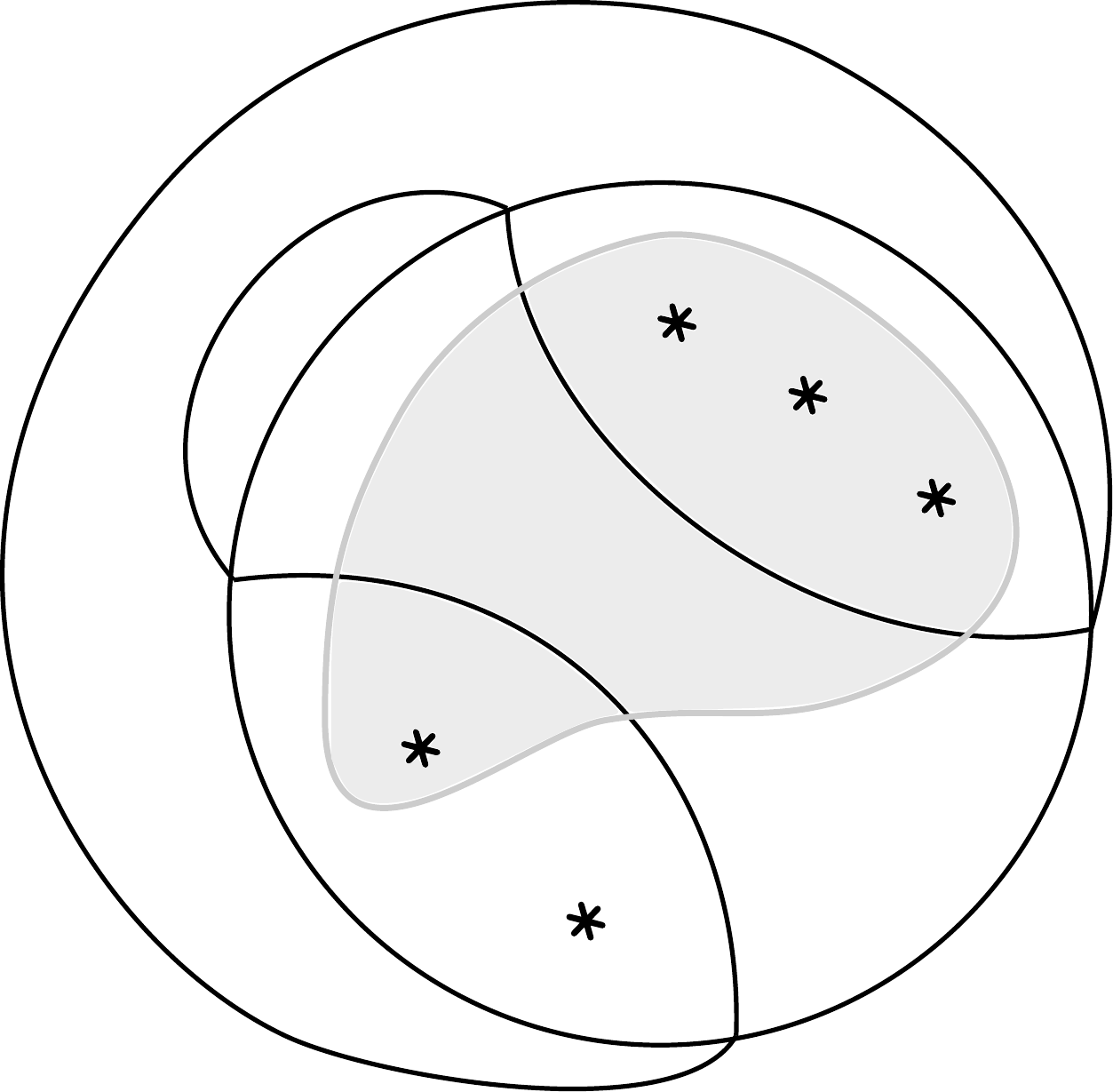}}%
    \put(0.42867038,0.50365504){\color[rgb]{0.50196078,0.50196078,0.50196078}\makebox(0,0)[lt]{\lineheight{1.25}\smash{\begin{tabular}[t]{l}$D$\end{tabular}}}}%
    \put(0.51480495,0.10247886){\color[rgb]{0,0,0}\makebox(0,0)[lt]{\lineheight{1.25}\smash{\begin{tabular}[t]{l}$\rho$\end{tabular}}}}%
    \put(0,0){\includegraphics[width=\unitlength,page=2]{punctured_circle_c2.pdf}}%
    \put(0.79009959,0.27264685){\color[rgb]{0.50196078,0.50196078,0.50196078}\makebox(0,0)[lt]{\lineheight{1.25}\smash{\begin{tabular}[t]{l}$c'$\end{tabular}}}}%
  \end{picture}%
\endgroup%

        \caption{The curve $c'$ in $\Omega(\partial W)$.\\}\label{fig:punctured_circle_c1}
    \end{minipage}
\end{figure}

Let $c'$ be the boundary of a disk with $K$ punctures that contains $D$ and $\rho$ and that avoids intersecting the arc $e$ by passing through $R$; see Figure~\ref{fig:punctured_circle_c1}.
The curve $c'$ is  in $\Omega(\partial W)$ and is connected to $c$ in $\Dis(Z, W)$ by the path $c'$, $c' \cup d$, $d$, $d \cup c$, $c$. Since $c'$ satisfies $M(c')=M(c)-1$,  we have that $\partial W$ is connected to  $c'$, and hence $\partial W$ is connected to every curve in $\Omega(\partial W)$ by induction. In particular, $\partial W$ is connected to $\phi (\partial W)$ for any Dehn twist generator $\phi \in X \sqcup X^{-1}$, because $\phi(\partial W) \in \Omega(\partial W)$.
\end{proof}

\section{The thick of order 2 case}\label{section:thick of order 2 case}

We now address those graphs of multicurves that do not satisfy the hypotheses of Theorem~\ref{theorem:hyperbolic_graphs_of_multicurves}  (hyperbolicity), Theorem~\ref{theorem:relatively_hyperbolic_graphs_of_multicurves} (relative hyperbolicity), or Theorems~\ref{theorem:thick case no separating annulus} or \ref{theorem:thick_punctured_disc_case} (thickness of order~$1$).
These graphs all satisfy the hypotheses of Theorem~\ref{theorem:thick_of_order_2_case} below, and we will prove that they are thick of order at most~2.
It is possible that a graph of multicurves satisfying the hypotheses of Theorem~\ref{theorem:thick_of_order_2_case} also satisfies the hypotheses of one of Theorem~\ref{theorem:thick case no separating annulus} or Theorem~\ref{theorem:thick_punctured_disc_case}.
This is no contradiction, as thick of order~1 is a case of thick of order at most~2.

\begin{theorem}\label{theorem:thick_of_order_2_case}
Let $\G(S)$ be a hierarchical graph of multicurves. Suppose there exist $Z,W \in \Wit\bigl(\G(S)\bigr)$ such that
\begin{itemize}
    \item $Z$ and $W$ are disjoint;
    \item $Z^c$ and $W^c$ are connected;
    \item  $S \sminus (Z \cup W)$ is a non-empty disjoint union of punctured annuli;
\end{itemize}
Then $\G(S)$ is thick of order at most 2.
\end{theorem}

Our proof of Theorem \ref{theorem:thick_of_order_2_case} is more involved than our previous proofs of thickness, so we will delay the proof until after we have collected a number of preliminary results. From now on, $\G(S)$ will be a hierarchical graph of multicurves whose set of witnesses, $\mf{S} = \Wit\bigl(\G(S)\bigr)$, satisfies the hypotheses of Theorem \ref{theorem:thick_of_order_2_case}.   Since the order of thickness is a quasi-isometry invariant, it will suffice to prove $\ksep_\mf{S}(S)$ is thick of order at most~$2$.  Fix a pair of co-connected witnesses, $Z,W \in \mf{S}$, that satisfy the hypotheses of Theorem \ref{theorem:thick_of_order_2_case}.

If there exist co-connected witnesses $Z ' \subseteq Z$ and $W' \subseteq W$ that do  not satisfy the hypotheses of Theorem~\ref{theorem:thick_of_order_2_case}, then $Z'$ and $W'$ must satisfy the hypotheses of Theorem~\ref{theorem:thick case no separating annulus}, so we have that $\G(S)$ is thick of order~1 (and hence thick of order at most~2). Accordingly, we will assume that $Z$ and $W$ are minimal in the sense that if there exist any $Z' \subseteq Z$ and $W' \subseteq W$ such that $Z'$ and $W'$ are co-connected witnesses, then $Z'=Z$ and $W' =W$. Moreover, we will assume that $S$ has genus at least~1. In the genus~0 case, the only possibility is that $Z$ and $W$ are both punctured disks separated by a punctured annulus, and thus $\G(S)$ is thick of order~1 by Theorem~\ref{theorem:thick_punctured_disc_case}.

Over the course of our proof, we will need to keep track of the punctures of $S$ that are contained in $S \sminus (Z \cup W)$.

\begin{definition}
    For our fixed $Z$ and $W$, we call a puncture $\rho$ of $S$ \emph{intermediate} if $\rho$ is contained in $S \sminus (Z \cup W)$. We say a punctured annulus or punctured disk $Y \subseteq S$ is  \emph{intermediate} if each puncture in $Y$ is intermediate.  Figure \ref{figure:intermediate_annulus} shows  two examples of intermediate punctured annuli. 
\end{definition}

\begin{figure}[ht]
    \centering
    \def\svgscale{1}
    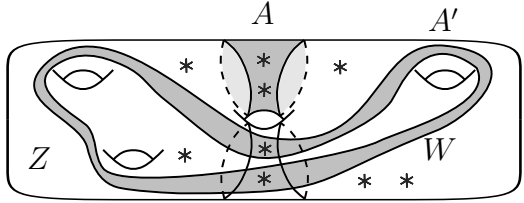
    \caption{The shaded annuli $A$ and $A'$ are intermediate punctured annuli for the depicted $Z$ and $W$. The punctures contained in $A$ and $A'$ are intermediate punctures, while the punctures contained in $Z$ and $W$ are not intermediate.}
    \label{figure:intermediate_annulus}
\end{figure}

Let $\Sigma$ be the surface obtained from $S$ by filling in all the intermediate punctures of $S$.
If $r$ is the number of intermediate punctures, the surface $\Sigma$ is fixed in the homeomorphism class of $S_{g,p-r}$.
The natural inclusion of $S$ into $\Sigma$ induces a map from curves on~$S$ to (possibly inessential) curves on~$\Sigma$.  As an aid to the reader, we will always use Roman letters to denote multicurves on the original surface~$S$ and Greek letters to denote multicurves on the surface~$\Sigma$ obtained by filling in the intermediate punctures of~$S$.

\begin{definition}[The puncture-filling map]
Let $F \colon \C(S) \to \C(\Sigma)\cup T$ be the map induced by filling in the punctures of $S \sminus (Z\cup W)$, where $T$ is the set of homotopy classes of curves on $\Sigma$ that bound a disk or a once punctured disk. We also use $F$ to denote the map $F \colon S \to \Sigma$ that fills in the punctures of $S \sminus (Z \cup W)$ with points.
\end{definition}

As in the proof of Theorem \ref{theorem:thick case no separating annulus}, let $\PDis(Z,W)$ be the subgraph of $\Dis(Z,W)$ spanned by the  set of vertices $m \in \Dis(Z,W)$ such that there exists  $f\in \pmcg(S)$ (instead of in $\mcg(S)$) so that $S \sminus m$ has one component containing $f(Z)$ and another component containing $f(W)$. Lemma~\ref{lemma: FO orbit}  below shows, in particular, that for  any multicurve $m \in \PDis(Z,W)$,  $F(m)$ contains an essential curve on~$\Sigma$.
By abuse of notation, we will use $F(m)$ to denote the essential multicurve on~$\Sigma$ obtained by removing any of the homotopy classes  in  $T$  from $F(m)$, if present.

\begin{definition}
Let $\FO$ denote the set of multicurves $\mu$ on $\Sigma$ such that there exists $m \in \PDis(Z,W)$ with $\mu = F(m)$. For a multicurve $\mu \in \FO$, we shall abuse notation and use $F^{-1}(\mu)$ to denote the full subgraph of $\PDis(Z,W)$ spanned by the set $\{m \in \PDis(Z,W) : F(m) = \mu\}$. 
\end{definition}

\begin{lemma} \label{lemma: FO orbit}
Let $m \in \PDis(Z,W)$ and $f \in \pmcg(S)$ so that $f(Z)$ and $f(W)$ are contained in distinct components of $S \sminus m$. Then:
\begin{enumerate}
    \item $\Sigma \sminus F(m) = F(f(Z)) \sqcup F(f(W))$;
    \item $F(m)$ is in the $\pmcg(\Sigma)$-orbit of $F(\partial Z)$.
\end{enumerate}
In particular, every element of $\FO$ is in the $\pmcg(\Sigma)$\hyp{}orbit of $F(\partial Z)$.

\end{lemma}

\begin{proof}
Since each component of $S \sminus (f(Z) \cup f(W))$ is an intermediate punctured annulus, every curve of $m$ which is not in $\partial f(Z)$ either bounds a punctured disk whose punctures are all intermediate or bounds an intermediate annulus with a curve of $\partial f(Z)$. Thus  $F(m) \subseteq F(f(\partial Z))$.
Moreover, if $F(m)$ was not all of $F(f(\partial Z))$, then $\Sigma \sminus F(m)$ would not separate $F(f(Z))$ and $F(f(W))$.
Then $m$ could not have separated $f(Z)$ and $f(W)$, which is a contradiction. Hence $F(m)=F(f(\partial Z))$.
By a similar argument, $F(m) = F(f(\partial W))$. This implies $\Sigma \sminus F(m) = F(f(Z)) \sqcup F(f(W))$.

For the second item, recall from Proposition~\ref{proposition:capping commutes} that $F \colon S \to \Sigma$ induces a homomorphism $F_* \colon \pmcg(S) \to \pmcg(\Sigma)$ so that $F(m) = F(f(\partial Z)) = F_*(f)(F(\partial Z))$.
\end{proof}

Unlike in the case of Theorem \ref{theorem:thick case no separating annulus} where $\G(S)$ is thick of order~$1$, in the present case we will find that $\PDis(Z,W)$ is not connected. However, we will use the connected components of $\PDis(Z,W)$ to build thick of order at most~$1$ subsets of $\ksep_\mf{S}(S)$. We find that the connected components of $\PDis(Z,W)$ are encoded by the fibers of the  puncture\hyp{}filling map~$F$, as follows.

\begin{lemma} \label{lemma:connectedcomponents}
The map $\mu \mapsto F^{-1}(\mu)$ is a bijection from $\FO$ to the set of connected components of $\PDis(Z,W)$.
\end{lemma}

\begin{proof}
Our proof has two steps. First we show that each connected component of $\PDis(Z,W)$ is contained in some $F^{-1}(\mu)$,  then we  prove each $F^{-1}(\mu)$ is connected.

\begin{claim}\label{claim:connected components have the same image}
If $m, n \in \PDis(Z,W)$ are joined by an edge then $F(m)=F(n)$, and hence $m$ and $n$ are both in $F^{-1}(\mu)$ for some $\mu \in \FO$.
\end{claim}

\begin{proof}
Without loss of generality, $n$ is obtained from $m$ by adding a curve $a$.
Let $f \in \pmcg(S)$ such that $f(Z)$ and $f(W)$ are contained in different components of $S \sminus n$. Thus, every curve of $n$ that is not a curve of $\partial f(Z)$ or $\partial f(W)$ must be contained in $S \sminus (f(Z) \cup f(W))$.
Since $f$ fixes each puncture of $S$, each component of $S \sminus (f(Z) \cup f(W))$ is an intermediate punctured annulus. Therefore $a$ either bounds an intermediate punctured disk  or bounds an intermediate punctured annulus with a curve of $m$.  Hence, $F(a)$ must either be homotopically trivial or be isotopic to  a curve in $F(m)$, and we have that $F(n)$ must equal $F(m)$.
\end{proof}

Inductively,  Claim~\ref{claim:connected components have the same image} implies that any two vertices that are contained in the same connected component of $\PDis(Z,W)$ have the same image under $F$.
We will now establish that each $F^{-1}(\mu)$ is connected.
We want to show that if $m$ and $n$ are two vertices of $\PDis(Z,W)$ so that $F(m)=F(n)=\mu$, then they are connected by a path in $F^{-1}(\mu)$.
We will use  induction on the intersection number $i(m, n)$.
The following claim is the base case of the induction.

\begin{claim}\label{claim:pull-back is connected: base case}
Let $m$, $n$ be vertices in $F^{-1}(\mu)$ so that $i(m,n)=0$.
Then there exists a path joining $m$ and $n$ in $F^{-1}(\mu)$. 
\end{claim}

\begin{proof}
Since $m$ and $n$ are vertices of $\PDis(Z, W)$, there exist $f,g \in \pmcg(S)$ so that $f(Z)$ and $f(W)$ are contained in different components of $S \sminus m$ and $g(Z)$ and $g(W)$ are contained in different components of $S \sminus n$.

If $S \sminus m$ has more than two components, we may remove curves until $S \sminus m$ has exactly two components, one containing $f(Z)$ and another containing $f(W)$.
This procedure does not change that $m$ is in $F^{-1}(\mu)$, nor increase $i(m,n)$.
We can similarly ensure that $S \sminus n$ has exactly two components.
In particular, no component of $S \sminus m$ or $S \sminus n$ is an intermediate punctured disk.

Since $F(m \cup n) = \mu$ and $i(m,n)=0$, every curve of $m$ that is not a curve of $n$ must bound an intermediate punctured annulus with some curve of $n$.
Define $K(m, n)$ to be the total number of punctures contained in the intermediate punctured annuli components of $S \sminus (m \cup n)$.

If $K(m,n)=0$, then $m=n$ and we are done.
We hence assume that $K(m, n)>0$.
If we can find a vertex $m' \in \PDis(Z,W)$ such that
\begin{itemize}
\item $F(m')=\mu$ with $i(m',n) =0$,
 \item $S \sminus m'$  has exactly two components,
\item $m'$ is connected to $m$ in $F^{-1}(\mu)$,
\item $K(m',n)<K(m,n)$,
\end{itemize}
then induction on $K(m,n)$ will show that $m$ is connected to $n$ in  $F^{-1}(\mu)$.

Now, we know that $S \sminus m$ has exactly two components: one, $U$, containing $f(Z)$ and one, $V$, containing~$f(W)$.
At least one of $U$ or $V$ must contain a curve of $n$ as a non\hyp{}peripheral curve, for if this were not the case, then every curve of $n$ would be isotopic to some curve of $m$ and $K(m,n)$ would be~0.
Without loss of generality, say that $U$ contains a curve $d$ of~$n$.
The curve $d$ is separated from a curve $c$ of $m$ by an intermediate punctured annulus.

If the intermediate punctured annulus between $d$ and $c$ contains exactly one puncture, let $c' = d$. Otherwise, let $c'$ be any curve on this annulus that cobounds exactly one intermediate puncture with $c$.

Let $m' = (m \cup c') \sminus c$. 
We can immediately check several of the conditions we want $m'$ to satisfy.
We have that $F(m')=F(m)=\mu$ with $i(m',n)=0$, that $S \sminus m'$ has exactly two components, and that $K(m',n)<K(m,n)$.
We still need to check that $m'$ is a vertex of $\PDis(Z, W)$ and that $m'$ is connected to $m$ by a path in $F^{-1}(\mu)$.
We will achieve this by proving that both $m'$ and $m \cup c'= m' \cup c$ are vertices of $\PDis(Z,W)$.

Recall, $S \sminus (Z \cup W)$ is a collection of disjoint intermediate punctured annuli $A_1,\dots, A_k$. Label each intermediate puncture of $S$ by the $A_i$ that contains it.  This labeling is $\pmcg(S)$-invariant as the elements of $\pmcg(S)$ fix each puncture of $S$. 
Without loss of generality, assume the puncture cobounded by $c'$ and $c$ is labeled by $A_1$. 

Since $f(Z)$ and $f(W)$ are respectively contained in the two components of $S \sminus m$, the multicurve $f^{-1}(m)$  has components $a_1,\dots, a_\ell$ so that $a_i$ is a curve on $A_i$ (possibly peripheral in the annulus) for each $i\in \{1,\dots,k\}$ and $a_i \in \partial Z \cap \partial W$ for each $i \in \{k+1,\dots,\ell\}$ (if $\ell =k$, then all $a_i$ are as in the first case). Hence we can find curves $u_1,\dots,u_\ell$ in $U$ (possibly peripheral in $U$) and $v_1,\dots, v_\ell$ in $V$ (possibly peripheral in $V$) so that 
\begin{itemize} 
    \item  $u_i$ and $v_i$ cobound a punctured annulus that contains exactly all the intermediate punctures labeled $A_i$ for each $i \in \{1, \dots, k\}$,
    \item  $u_i = v_i = f(a_i)$ for $i\in \{k+1,\dots, \ell\}$.
\end{itemize}
Note, the intermediate annuli cobounded by $u_i$ and $v_i$ in the first case might not be $f(A_i)$. 
We can moreover choose $u_1$ and $v_1$ so that $c'$ and $c$ are each either equal to $u_1$ and $v_1$ or contained in the intermediate annulus between $u_1$ and $v_1$. Figure \ref{figure:finding_h} demonstrates how finding $u_i$ and $v_i$ can be facilitated by gathering all the intermediate punctures of the same type around a unique curve in $m$.

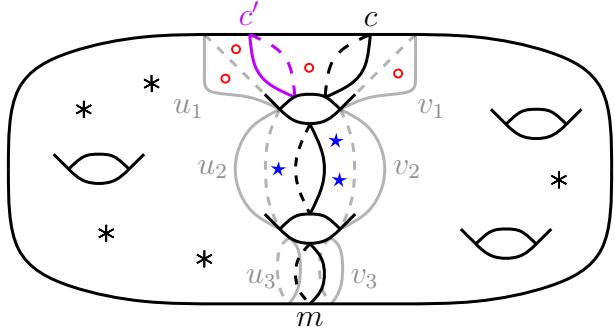
\begin{figure}[ht]
   \centering
   \definecolor{Purple}{RGB}{190,0,255}
   \definecolor{Grey}{gray}{.55}
   \begin{tikzpicture}
    \node[anchor=south west,inner sep=0] (image) at (0,0) {\includegraphics[scale=1.5]{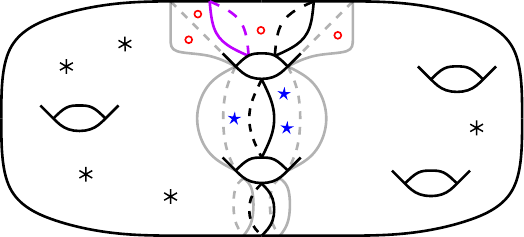}};
    \begin{scope}[x={(image.south east)},y={(image.north west)}]
       
       \node at (.5,-.05) {$m$};
       \node at (.42,.1) {\color{Grey}$u_3$};
       \node at (.59,.1) {\color{Grey}$v_3$};
       \node at (.6,1.05) {$c$};
       \node at (.4,1.07) {\color{Purple}$c'$};
       \node at (.3,.72) {\color{Grey}$u_1$};
       \node at (.7,.72) {\color{Grey} $v_1$};
       \node at (.34,.5) {\color{Grey}$u_2$};
       \node at (.66,.5) {\color{Grey}$v_2$};
    \end{scope}
    \end{tikzpicture}
    \caption{The red $\circ$ punctures are labeled by the annulus $A_1$, the blue $\star$ punctures are labeled by the annulus $A_2$. The black $\ast$ punctures are not intermediate.}
    \label{figure:finding_h}
\end{figure}

Now, $S \sminus (u_1 \cup \dots \cup  u_\ell \cup v_1 \cup \dots \cup v_\ell)$ contains $k$ components that are intermediate annuli, each of which contains exactly  all the intermediate punctures labeled by $A_i$. Hence, the ``change of coordinates principle'' \cite[Section 1.3]{primer} says there is some $h \in \pmcg(S)$ so that the components of $S \sminus (u_1 \cup \dots \cup  u_\ell \cup v_1 \cup \dots \cup v_\ell)$ are exactly $h(A_1), \dots, h(A_k)$ plus $h(Z)$ and $h(W)$. By construction, each curve of $m' \cup c$ is either contained in some $h(A_i)$ (possibly peripheral in $h(A_i)$) or is equal to a curve in $\partial h(Z) \cap \partial h(W)$. Hence, $h(Z)$ and $h(W)$ are contained in two different components of $S \sminus (m' \cup c)$. Similarly, $h(Z)$ and $h(W)$ are each contained in different components of $S\sminus m'$.  Thus, $m'$ and $m'\cup c$ are vertices of $\PDis(Z,W)$.
This connects $m'$ to $m$ by a path in $F^{-1}(\mu)$, completing our proof by induction.
\end{proof}

\begin{claim}\label{claim:pull-back is connected: induction}
Let $m$, $n$ be any two vertices in $F^{-1}(\mu)$.
Then there exists a path joining $m$ and $n$ in $F^{-1}(\mu)$.
\end{claim}

\begin{proof}
We will use an induction on $i(m, n)$.
The base case is when $i(m, n)=0$, and has just been verified in Claim~\ref{claim:pull-back is connected: base case}.
Now suppose $i(m, n)>0$.

Since we know that $m$ and $n$ both map to $\mu$ when we fill in the punctures of $S \sminus (Z\cup W)$, we must have an innermost punctured bigon $B$ bounded by an arc of $m$ and an arc of $n$.
We can surger $m$ across $B$  as shown in Figure~\ref{figure:surgery_across_puncture} to get a multicurve $m'$ such that:
\begin{itemize}
    \item $i(m', n) < i(m,n)$
    \item $i(m', m)=0$
    \item $F(m')=\mu$.
\end{itemize}

\begin{figure}[ht]
    \centering
    \def\svgscale{1}
    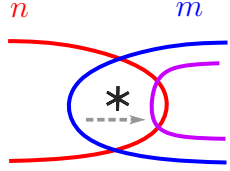
    \caption{The surgery of $m$ to $m'$ across a punctured bigon.}
    \label{figure:surgery_across_puncture}
\end{figure}

\noindent 
We claim that $m'$ is a vertex of $\PDis(Z,W)$. Let $c'$ be the single curve of $m'$ that is not a curve of $m$. Let $A$ be the punctured annulus cobounded by $c'$ and a curve of $m$. The punctures of $A$ are precisely the punctures of the innermost bigon $B$ between $m$ and $n$. Since $F(m) = F(n)$, any puncture contained in the  bigon $B$ is intermediate. Hence $A$ is an intermediate annulus. Let $\Pi$ be the set of punctures of $A$, and  let $f \in \pmcg(S)$ so that $f(Z)$ and $f(W)$ are  contained in different components of $S \sminus m$.

Since $A$ is intermediate, there exists a curve $c$ so that \begin{itemize}
    \item $c$ is contained in $S\sminus f^{-1}(m)$;
    \item $c$ is disjoint from both $Z$ and $W$;
    \item $c$  and a curve of $f^{-1}(m)$ cobound an intermediate annulus  whose punctures are exactly $\Pi$.
\end{itemize}
For such a curve $c$, the complement of $c$ in $S \sminus f^{-1}(m)$ is homeomorphic to the complement of $c'$ in  $S \sminus m$. Hence, there is an element $h \in \pmcg(S)$ so that $h(f^{-1}(m) \cup c) = m \cup c'$. Since $c$ is disjoint from both $Z$ and $W$, we have that $h(Z)$ and $h(W)$ are each contained in a different component of $S \sminus (m \cup c')$. This implies that $h(Z)$ and $h(W)$ are in different component of $S \sminus m'$ as well. 
 Therefore, $m'$ is a vertex of $\PDis(Z, W)$ and thus also a vertex of $F^{-1}(\mu)$.

Now, $m'$ is connected to $m$ in $F^{-1}(\mu)$ because $i(m',m) =0$ (Claim~\ref{claim:pull-back is connected: base case}) and $m'$ is  connected to $n$ in $F^{-1}(\mu)$ by   the induction hypothesis. Therefore $m$ is connected to~$n$ in $F^{-1}(\mu)$ as desired.
\end{proof}

Claim \ref{claim:connected components have the same image} implies each connected component of $\PDis(Z,W)$ is contained in $F^{-1}(\mu)$ for some $\mu \in \FO$, while Claim \ref{claim:pull-back is connected: induction} implies that for each $\mu \in \FO$, $F^{-1}(\mu)$ is contained in a single connected component of $\PDis(Z,W)$. Thus, we have the conclusion of Lemma~\ref{lemma:connectedcomponents}. 
\end{proof}

As shown in Lemma \ref{lemma:connected_components_are_thick_of_order_1}, each connected component of $\PDis(Z,W)$ will produce a thick of order at most~$1$ subset of $\ksep_\mf{S}(S)$. Since the connected components of $\PDis(Z,W)$ are precisely the fibers of elements of $\FO$ under the puncture\hyp{}filling  map $F$, these thick of order at most~$1$ subsets are indexed by the elements of $\FO$.

\begin{definition}
Given $\mu\in\FO$, define $\mc{X}(\mu)=\bigcup_{m\in F^{-1}(\mu)}P(m)$.
\end{definition}

\begin{lemma}\label{lemma:Xs are thick of order 1}
For each $\mu\in\FO$, $\mc{X}(\mu)$ is thick of order at most~1. \qed
\end{lemma}

To prove that $\ksep_\mf{S}(S)$ is thick of order at most~$2$, we need to be able to thickly chain together any two of the the thick of order at most~$1$ subsets $\mc{X}(\mu)$.  
Echoing the use of the connected subsets of $\PDis(Z,W)$ to understand the chaining of thick of order~$0$ subsets, we will use the connectivity of the following graph to guide the thick chaining of the $\mc{X}(\mu)$'s.

\begin{definition}
Let $\F$ be the graph with:
\begin{itemize}
    \item a vertex for every multicurve in $\FO$
    \item an edge between two vertices if they intersect at most 4 times on $\Sigma$.
\end{itemize}
\end{definition}

\begin{lemma}\label{lemma:F graph is connected}
The graph $\F$ is connected.
\end{lemma}

\begin{proof}
We will use Lemma~\ref{lemma:putman_trick}, with the group $\pmcg(\Sigma)$. Fix  $\mu \in \FO$. By Lemma \ref{lemma: FO orbit}, every vertex of $\F$ is in the $\pmcg(\Sigma)$-orbit of $\mu$.
 
For the second hypothesis of Lemma~\ref{lemma:putman_trick}, arrange a collection of standard generating curves for $\pmcg(\Sigma)$ so that each generating curve intersects the multicurve $\mu$ at most twice. This can be done  similarly to the examples in Figure~\ref{figure:gen_twists_separated_by_pants_Humphries}. Let $X$ be the set of Dehn twists around this collection of standard generating curves. For any $\phi \in X \sqcup X^{-1}$, we have that $\mu$ and $\phi(\mu)$ are joined by an edge of $\F$ because $i(\mu, \phi(\mu)) \leq 4$. Thus, $\F$ is connected. 
\end{proof}

Since $\F$ is connected, the remaining step for proving $\ksep_\mf{S}(S)$ is thick of order at most $2$ is to show that an edge in $\F$ between $\mu$ and $\nu$ implies $\mc{X}(\mu)$ and $\mc{X}(\nu)$ have infinite diameter coarse intersection.

\begin{proposition} \label{proposition: adjacent thick of order 1 pieces have thick intersection}
There exists $K$ depending only on $S$ such that if $\mu$ and $\nu$ are adjacent vertices of $\F$, then $\mc{N}_K(\mc{X}(\mu))\cap \mc{N}_K(\mc{X}(\nu))$ has infinite diameter. 
\end{proposition}

\begin{proof}

The proof of the proposition has three steps.
First we show that if $\mu,\nu \in \F$ are joined by an edge, then we can lift $\mu$ and $\nu$ to multicurves $m \in F^{-1}(\mu)$ and $n \in F^{-1}(\nu)$ on $S$ that intersect at most $4$ times (Claim \ref{claim:adjecent vertices lift to bounded intersection number}).
Next, we show that we can add curves to each of $m$ and $n$  to produce multicurves $x \in P_\mf{S}(m) \subseteq \mc{X}(\mu)$ and $y \in P_\mf{S}(n) \subseteq \mc{X}(\nu)$  whose intersection number is bounded by a number $N$ depending only on $S$ (Claim \ref{claim:promoting to K(S)}).  Because there are only finitely many $\pmcg(S)$-orbits of pairs of multicurves intersecting at most $N$ times, the bound $N$ on the intersection number between $x$ and $y$ produces a bound on the distance between $x$ and $y$ that depends only on $S$.
Finally, we find a pseudo-Anosov element $f \in \pmcg(S)$ that stabilizes both $\mc{X}(\mu)$ and $\mc{X}(\nu)$ (Claims \ref{claim:push_fixes_X(mu)} and \ref{claim:push_contains_a_pseudo-Anosov}). This implies the $f$-orbit of $x$ is an infinite diameter subset of the coarse intersection of $\mc{X}(\mu)$ and $\mc{X}(\nu)$.

\begin{claim}\label{claim:adjecent vertices lift to bounded intersection number}
Let $\mu$ and $\nu$ be adjacent vertices in $\F$. There exist $m \in F^{-1}(\mu)$ and $n \in F^{-1}(\nu)$ so that $i(m,n) \leq 4$ and $S \sminus m$ and $S \sminus n$ both have exactly two components.
\end{claim}

\begin{proof}
By Lemma \ref{lemma: FO orbit}, there exist $f,g \in \pmcg(S)$ so that $\mu = F(f(\partial Z))$ and $\nu = F(g(\partial Z) )$.  Let $m = f(\partial Z)$ and $n_1 = g(\partial Z)$. Note, $m,n_1 \in \PDis(Z,W)$ as they are in the $\pmcg(S)$-orbit of the boundary of $Z$. If $i(m,n_1) \le 4$, then we are finished, with $m =f(\partial Z)$ and $n = n_1 = g(\partial Z)$. Assume that $i(m,n_1)) > 4$. This implies that $m$ and $n_1$ form an innermost bigon that contains only intermediate punctures, because $i(\mu,\nu) \leq 4$ and $F(m) =\mu$ and $F(n_1) = \nu$. By performing a surgery across this bigon similarly to Figure \ref{figure:surgery_across_puncture}, there exists a multicurve $n_2$ so that $F(n_2) = \nu$ and $i(m,n_2) < i(m,n_1)$. Since the bigon we surgered $n_1$ over contained only intermediate punctures, then as in Claim~\ref{claim:pull-back is connected: induction}, there exists $h \in \pmcg(S)$ so that $S \sminus n_2$ has a component that contains $h(Z)$ and another that contains $h(W)$.
Thus, $n_2 \in \Dis(Z,W)$. If $i(m,n_2) \leq 4$, then we are finished with $m = f(\partial Z)$ and $n= n_2$. Otherwise, we can repeat this process at most $i(m,n_1) - 4$ times until we produce a multicurve $n_k$ so that $i(m,n_k) \leq 4$, $F(n_k) = \nu$, and $n_k \in \PDis(Z,W)$. We are then finished with $m = f(\partial Z)$ and $n = n_k$.
\end{proof}

\begin{claim}\label{claim:promoting to K(S)}
There exists $D$ depending only on $\xi(S)$ so that the following holds for all adjacent vertices $\mu$ and $\nu$ in $\F$. If $m \in F^{-1}(\mu)$ and $n \in F^{-1}(\nu)$ are the vertices of  $\PDis(Z,W)$  obtained in Claim~\ref{claim:adjecent vertices lift to bounded intersection number}, then there exist $x \in P_\mf{S}(m)$ and $y \in P_\mf{S}(n)$ such that $i(x,y) \leq D$
\end{claim}

\begin{proof}
Let $ S \sminus m =U_1 \sqcup U_2 $ and $S \sminus n = V_1 \sqcup V_2$. Note, if $u_i$ is any multicurve defining a pants decomposition of $U_i$, for $i=1,2$, and $v_i$ is any multicurve defining a pants decomposition of $V_i$, then $m \cup u_1 \cup u_2$ and $n\cup v_1 \cup v_2$ are elements of $P_\mf{S}(m)$ and $P_\mf{S}(n)$ respectively. We will show that there always exist choices for $u_i$ and $v_i$ that ensure the intersection number of the two resulting multicurves is uniformly bounded.

Since $n$ is a multicurve and $i(m,n)\le4$, the intersection $n\cap U_i$ is a collection of at most four arcs and at most $\xi(U_i)$ curves.
The number of $\mcg(U_i)$-orbits of pants decompositions of $U_i$ is  bounded in terms of $\xi(S)$. Thus there exists a number $D_1=D_1(\xi(S))$, such that every $\mcg(U_i)$\hyp{}orbit of pants decompositions of~$U_i$ has a representative which intersects $n \cap U_i$ at most $D_1$ times.
For $i=1,2$, we let $u_i$ be any multicurve defining a pants decomposition of $U_i$ with $i(u_i,n\cap U_i) \leq D_1$ and then define $x = m \cup u_1 \cup u_2$. 

Now, $i(x,n)$ will be at most $2D_1 + 4$. Thus, we can repeat the argument above with  $x$ and the $V_i$ to produce a  multicurve $v_i$ defining a pants decomposition of $V_i$, and a number $D_2$ depending only on $\xi(S)$, such that $i(v_i, x \cap V_i) \leq D_2$.
Defining $y = n \cup v_1 \cup v_2$  we have $i(x,y) \leq 2D_2 + 2D_1 +4$. This depends only on $\xi(S)$ as desired.
\end{proof}

We now use Claims \ref{claim:adjecent vertices lift to bounded intersection number} and \ref{claim:promoting to K(S)} to finish the proof of Proposition \ref{proposition: adjacent thick of order 1 pieces have thick intersection}.

Let $\mu,\nu$ be adjacent vertices of $\F$ and let $m \in F^{-1}(\mu)$  and $n \in F^{-1}(\nu)$  be as specified in Claim \ref{claim:adjecent vertices lift to bounded intersection number}. Claim \ref{claim:promoting to K(S)} provides a number $D \geq 0$, depending only on~$\xi(S)$, and vertices $x \in P_{\mf{S}}(m)$, $y \in P_\mf{S}(n)$ with $i(x,y) \leq D$. Since there are finitely many $\pmcg(S)$\hyp{}orbits of pairs of vertices of $\ksep_\mf{S}(S)$ intersecting at most $D$ times, there exists $K \geq 0$ depending only on $S$ such that $d_{\ksep_\mf{S}(S)}(x,y) \leq K$. In particular, $x \in \mc{N}_K(\mc{X}(\mu)) \cap \mc{N}_K(\mc{X}(\nu))$.
We now want to apply an appropriate pseudo\hyp{}Anosov mapping class to $x$ to obtain an infinite diameter subset of $\mc{N}_K(\mc{X}(\mu)) \cap \mc{N}_K(\mc{X}(\nu))$.
This mapping class will need to stabilize $\mc{N}_K(\mc{X}(\mu)) \cap \mc{N}_K(\mc{X}(\nu))$.

Let $\mc{P}$ be the kernel of the homomorphism $F_*\colon \pmcg(S) \to \pmcg(\Sigma)$ induced by filling in the intermediate punctures of $S$. 
We claim that $\mc{P}$ fixes both $\mc{X}(\mu)$ and $\mc{X}(\nu)$ setwise and contains a pseudo-Anosov element.

\begin{claim}\label{claim:push_fixes_X(mu)}
For each $\alpha\in \FO$, the subgroup $\mc{P}$ preserves $\mc{X}(\alpha)$ setwise.
\end{claim}

\begin{proof}
Recall that $\mc{X}(\alpha)=\bigcup_{a\in F^{-1}(\alpha)}P_\mf{S}(a)$. If $f\in\mc{P}=\ker(F_*)$, then for all $a \in F^{-1}(\alpha)$: \[F(f(a))= F_*(f)(F(a))= F(a) = \alpha.\]
Hence $f(a)\in F^{-1}(\alpha)$ for all $f \in \mc{P}$. Since $f$ preserves $F^{-1}(\alpha)$, it also preserves $\mc{X}(\alpha)=\bigcup_{a\in F^{-1}(\alpha)}P(a)$.
\end{proof}

\begin{claim}\label{claim:push_contains_a_pseudo-Anosov}
There exists a pseudo\hyp{}Anosov mapping class in the subgroup $\mc{P}$.
\end{claim}

\begin{proof}
Choose an intermediate puncture $\rho$ of~$S$, and let $\mc{P}_\rho$ be the kernel of the induced map on mapping class groups when we fill in just the puncture~$\rho$.
This is a subgroup of~$\mc{P}$.
By \cite[Theorem~2']{kra}, the subgroup $\mc{P}_\rho$ contains a pseudo\hyp{}Anosov mapping class.
Hence, $\mc{P}$ contains a pseudo\hyp{}Anosov mapping class.
\end{proof}

To finish the proof of Proposition \ref{proposition: adjacent thick of order 1 pieces have thick intersection}, let $f \in \mc{P}$ be a pseudo-Anosov element. Since $f$ preserves $\mc{X}(\mu)$ and $\mc{X}(\nu)$ setwise, we have $f^k(x) \in \mc{N}_K(\mc{X}(\mu)) \cap \mc{N}_K(\mc{X}(\nu))$ for all $k \in \mathbb{N}$.
Since the $f$-orbit of $x$ has infinite diameter in $\ksep_\mf{S}(S)$ by Corollary~\ref{corollary:pA_are_undistorted}, we have that $\mc{N}_K(\mc{X}(\mu)) \cap \mc{N}_K(\mc{X}(\nu))$ must also have infinite diameter.
\end{proof}

\begin{proof}[Proof of Theorem \ref{theorem:thick_of_order_2_case}]
Since the order of thickness is a quasi-isometry invariant, it suffices to prove $\ksep_\mf{S}(S)$ is thick of order at most~$2$.

 Since the action of $\pmcg(S)$ on $\ksep_\mf{S}(S)$ is cobounded, there exists $D$ depending only on $S$ so that $\ksep_\mf{S}(S) \subseteq \bigcup_{\mu \in \F} \mc{N}_D(\mc{X}(\mu))$. By Lemma~\ref{lemma:Xs are thick of order 1}, $\mc{X}(\mu)$ is thick of order at most~$1$ for each $\mu \in \F$. Finally, if $\mu,\nu \in \F$, then there exists a path $\mu=\mu_0, \mu_1, \dots, \mu_k = \nu$ in $\F$ connecting $\mu$ to $\nu$ by Lemma~\ref{lemma:F graph is connected}. By Proposition~\ref{proposition: adjacent thick of order 1 pieces have thick intersection}, there exists $K$ depending only on $S$ so that $\mc{N}_K(\mc{X}(\mu_{i-1})) \cap \mc{N}_K(\mc{X}(\mu_{i}))$ has infinite diameter for each $i \in \{1,\dots,k\}$. Therefore, $\ksep_{\mf{S}}(S)$ is thick of order at most~$2$.
\end{proof}

\section{Examples of the classification} \label{section:applications}

We now discuss how our classification of  hyperbolicity, relative hyperbolicity and thickness applies to specific examples of hierarchical graphs of multicurves. In our first two examples, Theorem \ref{theorem:witnesses_determine_geoemetry}  recovers previously known results from the literature.

\begin{example}[The Pants Graph]
The vertices of the pants graph, $\mc{P}(S)$, are all multicurves that define pants decompositions of  $S$. Two multicurves $x,y \in \mc{P}(S)$ are joined by an edge if  there exist curves $\alpha \in x$ and $\beta \in y$ such that $(x\sminus\alpha) \cup \beta = y$ and $\alpha$ and $\beta$ intersect minimally on the complexity~$1$ component of $S \sminus (x \sminus\alpha)$. The witnesses for $\mc{P}(S)$ are all connected subsurfaces with complexity at least $1$.

The pants graph is known to be hyperbolic when $\xi(S) \leq 2$ \cite{Brock_Farb_Rank}, relatively hyperbolic when $\xi(S) = 3$ \cite{Brock_Masur_WP_Low_complexity}, thick of order~$1$ when $\xi(S) >3$ and $S\neq S_{2,1}$ \cite{Brock_Masur_WP_Low_complexity,Behrstock_Drutu_Mosher_Thickness}, and thick of order at most~$2$ when $S= S_{2,1}$ \cite{Brock_Masur_WP_Low_complexity, sultanthesis}. One can verify that these results match up precisely with the classification in terms of witnesses given in Theorem~\ref{theorem:witnesses_determine_geoemetry}.
\end{example}

\begin{example}[The separating curve graph and the Torelli graph]
The vertices of the separating curve graph,  $\sep(S)$,  are all the separating curves on $S = S_{g,p}$. Two separating curves are joined by an edge if they are disjoint. For surfaces with low complexity that contain separating curves, we ensure the separating curve graph is connected by putting an edge between any two curves that intersect four times for $S = S_{2,0}$  or $S_{1,2}$, or intersect twice for  $S = S_{0,4}$.  The witnesses for $\sep(S)$ are connected subsurfaces $ W\subseteq S$ such that each component of $S \sminus W$  contains no genus and at most  one puncture of~$S$.

When $S= S_{0,p}$, the separating curve graph is equal to the curve graph of $S$ and is therefore hyperbolic \cite{mm1}. Sultan showed that $\sep(S_{2,0})$ is also hyperbolic \cite{sultan_genus_two}.
Previous work of the authors completed the classification, showing that $\sep(S)$ is hyperbolic if $p \geq 3$ or $(g,p) \in \{(2,0),(2,1),(1,2)\}$ \cite{vokessep}, relatively hyperbolic if $p =0$  and $g \ge 3$ or $p=2$ and $g \ge 2$ \cite{russell}, and thick of order at most~$2$ if $p=1$  and $g \ge 3$ \cite{russell_vokes}.  These results again match up with the  classification in terms of witnesses given in Theorem~\ref{theorem:witnesses_determine_geoemetry} and employ special cases of the techniques employed here.

Closely related to the separating curve graph is the Torelli graph introduced by Farb and Ivanov for a closed surface $S=S_{g,0}$ \cite{farbivanov}. The \emph{Torelli graph}, $\mc{T}(S)$, is the graph whose vertices are all separating curves of $S$ plus all bounding pairs of non-separating curves (a pair of disjoint, non-separating curves $c_1$ and $c_2$ form a \emph{bounding pair} if $S \sminus (c_1 \cup c_2)$ is disconnected). There is an edge between two vertices of $\mc{T}(S)$ if the multicurves are disjoint.  The Torelli graph is connected for $g \geq 3$ because the separating curve graph is connected in those cases and every bounding pair of non-separating curves is disjoint from some separating curve. While combinatorially distinct, we find that the Torelli graph has the same coarse geometry as the separating curve graph. 

\begin{lemma}
If $S=S_{g,0}$, then $\Wit\bigl(\sep(S)\bigr) = \Wit\bigl(\mc{T}(S)\bigr)$. In particular, the inclusion of $\sep(S)$ into $\mc{T}(S)$ is a quasi-isometry when $g \geq 3$ and thus $\mc{T}(S)$ is relatively hyperbolic with peripherals quasi-isometric to $\C(S_{0,g+1}) \times \C(S_{0,g+1})$.
\end{lemma}

\begin{proof}
Since every vertex of $\sep(S)$ is also a vertex for $\mc{T}(S)$, we have $\Wit\bigl(\mc{T}(S)\bigr) \subseteq \Wit\bigl(\sep(S)\bigr)$. Now, let $W \in \Wit\bigl(\sep(S)\bigr)$. If $W$ is not a witness of $\mc{T}(S)$, then there exists a bounding pair of non-separating curves $c_1 \cup c_1$  so that $W$ is contained in a component of $S \sminus (c_1\cup c_2)$. Since each component of $S \sminus (c_1 \cup c_2)$ contains genus, this implies a component of $S \sminus W$ contains genus. However, this contradicts $W$ being a witness for $\sep(S)$. Therefore $ W \in \Wit\bigl(\mc{T}(S)\bigr)$ and we have  $\Wit\bigl(\sep(S)\bigr) = \Wit\bigl(\mc{T}(S)\bigr)$.

Since $\Wit\bigl(\sep(S)\bigr) = \Wit\bigl(\mc{T}(S)\bigr)$, Proposition \ref{proposition: k qi g} implies the inclusion  of $\sep(S)$ into $\mc{T}(S)$ is a quasi-isometry whenever $\mc{T}(S)$ is connected. The statement on relative hyperbolicity follows from the relative hyperbolicity of $\sep(S)$ when $S$ is closed; see~\cite{russell}.
\end{proof}

\begin{remark}
An anonymous referee communicated the following direct proof that $\sep(S)$ and $\mc{T}(S)$ are quasi-isometric. Consider a geodesic of $\mc{T}(S)$ with vertices $\alpha_0,\alpha_1, \dots, \alpha_n$. For each $i \in \{1,\dots, n\}$, $\alpha_{i-1}$ and $\alpha_{i+1}$ must intersect. Since each vertex of $\mc{T}(S)$ separates $S$, both $\alpha_{i-1}$ and $\alpha_{i+1}$ must be contained in the same component of $S \sminus \alpha_i$. Now, if $\alpha_i$ is a bounding pair, then there must be a separating curve $\beta_i$ that is in the component of $S \sminus \alpha_i$ that does not contain $\alpha_{i-1}$ and $\alpha_{i+1}$. Hence $\alpha_0, \dots, \alpha_{i-1}, \beta_i, \alpha_{i+1}, \dots, \alpha_n$ are also the vertices of a geodesic in $\mc{T}(S)$. Applying this inductively, any pair of separating curve in $\mc{T}(S)$ can be joined by a geodesic in $\mc{T}(S)$ whose vertices are all separating curves. This makes the inclusion $\sep(S) \to \mc{T}(S)$ an isometric embedding. Since every bounding pair is disjoint from some separating curve, this isometric embedding is a quasi-isometry. 
\end{remark}

\end{example}

Our third example is a family of complexes whose vertices are non-separating multicurves. This family contains the non-separating curve graph, the cut system graph, and a collection of graphs studied  by Hamenst\"adt.  

\begin{example}[Non-separating multicurve graphs] \label{example:non-separating}

Let $k \geq 1$ and $S= S_{g,p}$.  Let $\G(S)$ be any hierarchical graph of multicurves whose vertices are all  non-separating multicurves of $S$ with $k$ components. Examples of such graphs include the non-separating curve graph of $S$ ($k=1$), the graph of cut systems  introduced  by Hatcher and Thurston in~\cite{Hatcher_Thurston_complex} ($k=g$), and the collection of graphs studied by  Hamenst\"adt in \cite{Hamenstadt_non-separating}. 

 For any subsurface $Y \subseteq S$,  the total genus $g$ of $S$ is  equal to the the sum of the genus of $Y$, the genus of $S \sminus Y$, and the number  curves in the largest possible non-separating multicurve $\nu$ of $S$ where each curve of $\nu$ is isotopic to a curve of $\partial Y$.  Using this equation, one can show that a subsurface $Y$ is disjoint from   some vertex $\mu \in \G(S)$ if and only if  the genus of $Y$ is at most $g-k$. Hence, 
the  set of witnesses for $\G(S)$ is the set of all connected subsurfaces whose genus is at least $g - k +1$.
 Moreover, the coarse geometry of $\G(S)$ depends only on $g$, $p$ and~$k$.

Applying the results of this paper to $\G(S)$ we obtain Corollary~\ref{intor_corollary:Non-separating_curve_complexes}.
In particular, this recovers results of Hamenst\"adt for $k<g/2+1$~\cite{Hamenstadt_non-separating} and Li and Ma for the cut system graph of $S_{2,0}$~\cite{Li_Ma_Hatcher-Thurston}.
\end{example}

\introcor*

\begin{proof}
If $k < g/2 + 1$, then every witness of $\G(S)$ must have genus greater than $g/2$. Since any two subsurfaces of $S$ with genus  greater than $g/2$ must intersect, any two witnesses of $\G(S)$ must intersect. Thus, $\G(S)$ is hyperbolic by Theorem~\ref{theorem:hyperbolic_graphs_of_multicurves}.

Suppose $k = g/2 + 1$  and $p=0$.  Let $Z,W$ be a pair of disjoint, co-connected witnesses for $\G(S)$. Since $Z$ and $W$ are witnesses, they must each have genus at least~$g/2$. However, since $Z$ and $W$ are disjoint, each of them has genus exactly~$g/2$. Since $S$ is closed, this implies that $S \sminus Z = W$  and that $Z$ and $W$ are each homeomorphic to~$S^1_{g/2, 0}$. Thus, $\G(S)$ is relatively hyperbolic by Theorem \ref{theorem:relatively_hyperbolic_graphs_of_multicurves}, and the peripheral subsets are all quasi-isometric to $P_\mf{S}(\partial Z)$. 
Because  $Z$ and $W$  both have minimal genus for a witness for $\G(S)$, no proper subsurface of either $Z$ or $W$ is a witness for~$\G(S)$. Thus, Corollary~\ref{corollary:product regions} plus Proposition~\ref{proposition: k qi g}  implies that $P_\mf{S}(\partial Z)$ is quasi-isometric to $\C(S^1_{g/2,0}) \times \C(S^1_{g/2,0})$.

Now let $k = g/2 +1$ and $p>0$. There exists a pair of disjoint separating curves $c_1$ and $c_2$ so that $S \sminus (c_1 \cup c_2)$ has one component homeomorphic to an annulus with $p$ punctures and two components---$Z$ and $W$---both homeomorphic to $S^1_{g/2,0}$; see Figure~\ref{figure:Nonsep_order_2_case} for an example when $g=6$ and $k = 4$. Since $k = g/2+1$, both $Z$ and $W$ are witnesses for $\G(S)$. Thus, $Z$ and $W$ satisfy the hypotheses of Theorem~\ref{theorem:thick_of_order_2_case} and $\G(S)$ is thick of order at most~2. 
\begin{figure}[ht]
    \centering
    \def\svgscale{1}
    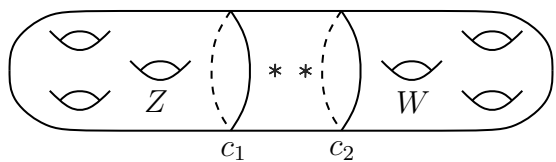
    \caption{An example of the subsurfaces $Z$ and $W$ for $g = 6$ and $k=g/2+1$.}
    \label{figure:Nonsep_order_2_case}
\end{figure}

Finally, suppose $k > g/2+1$. Let $h=g/2 - 1$ if $g$ is even or $(g-1)/2$ if $g$ is odd.  Every subsurface of $S$ with genus at least  $h$ is then a witness for $\G(S)$. Let $c_1$ and $c_2$ be a pair of disjoint separating curves on $S$ so that $S \sminus (c_1 \cup c_2)$ has two components---$Z$ and $W$---both homeomorphic to $S^1_{h,0}$ and a third component $Y$ that has genus at least~$1$; see Figure~\ref{figure:Nonsep_order_1_case} for an example when $g=5$ and $k = 3$.
Since  $Z$ and $W$ both have genus $h$, they are both witnesses of $\G(S)$. Thus, $Z$,$W$, and $Y$ satisfy the hypothesis of Theorem \ref{theorem:thick case no separating annulus} and $\G(S)$ is thick or order~$1$. \qedhere

\begin{figure}[ht]
    \centering
    \def\svgscale{1}
    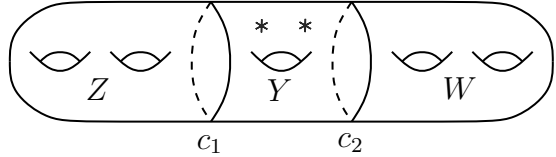
    \caption{An example of subsurfaces $Z$, $W$, and $Y$ for $g = 5$ and $k>g/2+1$.}
    \label{figure:Nonsep_order_1_case}
\end{figure}
\end{proof}

Studying the examples provided above may give the impression that our classification of thickness versus relative hyperbolicity can be simplified. For example, in the above examples, we only need to apply Theorem \ref{theorem:thick_punctured_disc_case} in the special case of  punctured spheres. Further, when one of the above graphs is relatively hyperbolic, the peripheral subsets are always the product of curve graphs of proper subsurfaces because each pair of disjoint witnesses does not contain any other witnesses for the graph. 

These observations are however purely the result of specific graphs and not indicative of a general phenomenon. In our final examples, we use Lemma \ref{lemma:ksep_is_hierarchical} to construct examples of hierarchical graphs of multicurves that require the full strength of  Theorem \ref{theorem:witnesses_determine_geoemetry}.  
We will also illustrate how the simple Euler characteristic criterion of Corollary~\ref{intro_corollary:closed_surface_case} does not work in general for surfaces with positive genus and punctures. The idea behind these examples is to force the set of witnesses to have the desired configuration by choosing  the minimal possible set of witnesses that contains the  mapping class group orbit of some specified subsurfaces.

\begin{example}[Relatively hyperbolic graph with complicated peripherals.]
Let $S = S_{2,6}$.  Let $c$ be a separating curve so that one component of $S \sminus c$ is homeomorphic to $S_{2,0}^1$ and the other is homeomorphic to $S_{0,6}^1$. Let $U \cong S_{1,0}^2$ be a subsurface of the component $S_{2,0}^1$ of $S \sminus c$ so that $\partial U$ contains $c$, and let $V \cong S_{0,3}^2$ be a subsurface of the $S_{0,6}^1$ component of $S\sminus c$ so that  $\partial V$ contains $c$; see Figure \ref{figure:Relatively_Hyperbolic_Case}.

\begin{figure}[ht]
    \centering
    \def\svgscale{1}
    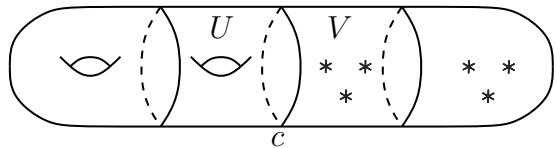
    \caption{The subsurfaces $U$ and $V$ in the complement of the curve $c$. }
    \label{figure:Relatively_Hyperbolic_Case}
\end{figure}

Let $\mf{S}$ be the set of all connected subsurfaces that contain a subsurface in the $\mcg(S)$-orbit of either $U$ or $V$.
Then $\ksep_{\mf S}(S)$ is a hierarchical graph of multicurves with $\mf{S}$ as its set of witnesses; see Lemma \ref{lemma:ksep_is_hierarchical}.
If $Z,W \in \mf{S}$ are a pair of disjoint, co-connected witnesses, then $Z$ and $W$ must be the components of $S \sminus \phi(c)$ for some $\phi \in\mcg(S)$. In particular, $S \sminus Z = W$, implying $\ksep_{\mf{S}}(S)$ is hyperbolic relative to the subsets $\{P_{\mf{S}}(\phi(c)) : \phi \in \mcg(S) \}$ by Theorem \ref{theorem:relatively_hyperbolic_graphs_of_multicurves}.

For $\phi \in \mcg(S)$, let $Z$ and $W$ be the connected components of $S \sminus \phi(c)$ so that $\phi(U) \subseteq Z$ and $\phi(V) \subseteq W$. Recall, for a subsurface $Y \subseteq S$, we define $\mf{S}_Y=\{Q \in \mf{S} : Q \subseteq Y\}$. By Corollary \ref{corollary:product regions}, $P_{\mf{S}}(\phi(c))$ is quasi-isometric to $ \ksep_{\mf{S}_{Z}}(Z) \times \ksep_{\mf{S}_{W}}(W)$. Since $Z$ strictly contains $\phi(U)$ and $W$ strictly contains~$\phi(V)$, both $\mf{S}_{Z}$ and $\mf{S}_{W}$ contain an infinite number of  witnesses.
In particular, the subsurface projection map from either $\ksep_{\mf{S}_{Z}}(Z)$ or  $\ksep_{\mf{S}_{W}}(W)$ to the curve graph of any proper subsurface of~$S$ is not a quasi\hyp{}isometry.

This also gives us an example of how the Euler characteristic criterion of Corollary~\ref{intro_corollary:closed_surface_case} does not work for general surfaces.
If $Z$ is the component of $S \sminus c$ containing~$U$ then $\vert \chi(Z) \vert < \frac{1}{2}\vert \chi(S) \vert$, and yet $\G(S)$ is relatively hyperbolic.
\end{example}

\begin{example}[Thick of order $1$ using Theorem \ref{theorem:thick_punctured_disc_case}]\label{example: using punctured disc thm}
Let $S = S_{2,4}$. Let $Z$ be any co-connected subsurface homeomorphic to $S_{2,0}^1$ and then let $W$ be any subsurface of $S \sminus Z$ that is homeomorphic to $D_3$; see Figure \ref{figure:Thick_1_disk_case}.  Let $\mf{S}$ be the set of connected subsurfaces that contain a subsurface in the $\mcg(S)$-orbit of either $Z$ or $W$. Since $Z$ and $W$ satisfy the hypotheses of Theorem \ref{theorem:thick_punctured_disc_case}, $\ksep_\mf{S}(S)$ is thick of order~$1$. Further, no pair of disjoint, co-connected elements of $\mf{S}$ will satisfy the hypotheses of Theorem~\ref{theorem:thick case no separating annulus}. 
\begin{figure}[ht]
    \centering
    \def\svgscale{1}
    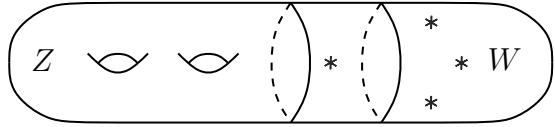
    \caption{The subsurfaces $Z$ and $W$  in Example~\ref{example: using punctured disc thm}.}
    \label{figure:Thick_1_disk_case}
\end{figure}
\end{example}

\begin{example}[Hyperbolic graphs with small witnesses]\label{example: hyperbolic graphs}
We finish with two examples of hyperbolic graphs of multicurves where there exists a witness $W$ which takes up ``less than half'' of the surface $S$ in the sense that $\vert \chi(W) \vert < \frac{1}{2} \vert \chi(S) \vert$.

For the first example let $S=S_{4,0}$.
Since $S$ is a closed surface, Corollary~\ref{intro_corollary:closed_surface_case} tells us that whenever $\mf{S}$ is a set of witnesses for a hierarchical graph of multicurves of~$S$ and whenever there exists a co-connected subsurface $W$ in $\mf{S}$ such that $\vert \chi(W) \vert < \frac{1}{2} \vert \chi(S) \vert$, then $\ksep_{\mf{S}}(S)$ is thick of order~1.
Our example will show why the assumption of co-connected is necessary here.
Let $W$ be a subsurface of $S$ homeomorphic to $S_{0,0}^4$, so that $S \sminus W$ is a disjoint union of four 1-holed tori; see Figure~\ref{figure:hyperbolicsmallwitness}.
Let $\mf{S}$ be the set of all connected subsurfaces containing some mapping class group translate of~$W$, and consider the hierarchical graph of multicurves $\ksep_\mf{S}(S)$.
Now, $2=\vert \chi(W) \vert < \frac{1}{2} \vert \chi(S) \vert=3$.
However, any two subsurfaces in $\mf{S}$ are forced to overlap, and hence $\ksep_\mf{S}(S)$ is hyperbolic by Theorem~\ref{theorem:hyperbolic_graphs_of_multicurves}.

\begin{figure}[ht]
    \centering
    \def\svgscale{.5}
\begingroup%
  \makeatletter%
  \providecommand\color[2][]{%
    \errmessage{(Inkscape) Color is used for the text in Inkscape, but the package 'color.sty' is not loaded}%
    \renewcommand\color[2][]{}%
  }%
  \providecommand\transparent[1]{%
    \errmessage{(Inkscape) Transparency is used (non-zero) for the text in Inkscape, but the package 'transparent.sty' is not loaded}%
    \renewcommand\transparent[1]{}%
  }%
  \providecommand\rotatebox[2]{#2}%
  \newcommand*\fsize{\dimexpr\f@size pt\relax}%
  \newcommand*\lineheight[1]{\fontsize{\fsize}{#1\fsize}\selectfont}%
  \ifx\svgwidth\undefined%
    \setlength{\unitlength}{297.11176178bp}%
    \ifx\svgscale\undefined%
      \relax%
    \else%
      \setlength{\unitlength}{\unitlength * \real{\svgscale}}%
    \fi%
  \else%
    \setlength{\unitlength}{\svgwidth}%
  \fi%
  \global\let\svgwidth\undefined%
  \global\let\svgscale\undefined%
  \makeatother%
  \begin{picture}(1,0.90911768)%
    \lineheight{1}%
    \setlength\tabcolsep{0pt}%
    \put(0,0){\includegraphics[width=\unitlength,page=1]{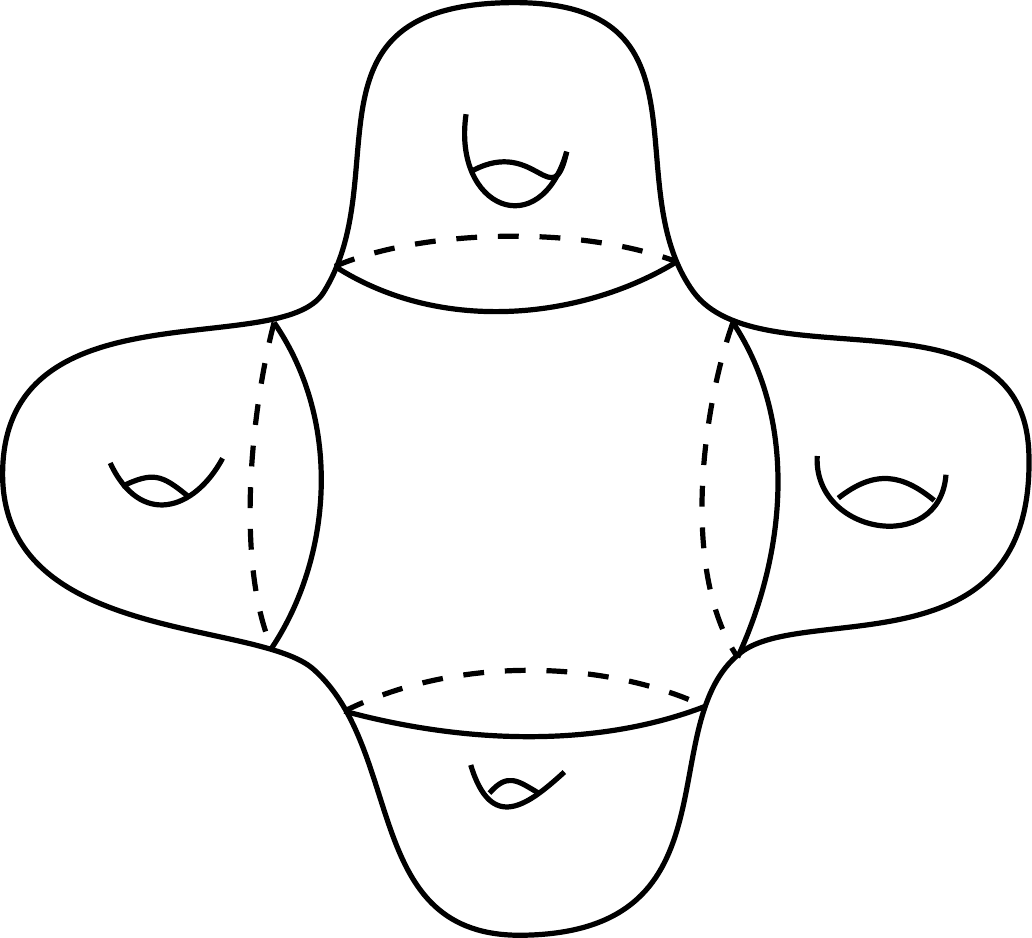}}%
    \put(0.45329174,0.40407826){\color[rgb]{0,0,0}\makebox(0,0)[lt]{\lineheight{1.25}\smash{\begin{tabular}[t]{l}$W$\end{tabular}}}}%
  \end{picture}%
\endgroup%

    \caption{A ``small'' witness for a hyperbolic hierarchical graph of multicurves.}
    \label{figure:hyperbolicsmallwitness}
\end{figure}

For the second example, the ``small'' witness will be co-connected, but in this case the surface will have both punctures and genus.
Let $S=S_{2,4}$ and let $Z$ and $W$ be subsurfaces homeomorphic to $S^1_{2,2}$ and $D_3$ respectively (so $W$ is as in Figure~\ref{figure:Thick_1_disk_case} above, but $Z$ has been extended to include two of the punctures of $S$).
Define $\mf{S}$ to be the set of all connected subsurfaces containing a subsurface in the $\mcg(S)$\hyp{}orbit of either $Z$ or~$W$.
We have $2=\vert \chi(W) \vert < \frac{1}{2} \vert \chi(S) \vert=3$.
However, once again any two subsurfaces in $\mf{S}$ must intersect, and so $\ksep_\mf{S}(S)$ is hyperbolic.

\end{example}

\bibliography{Bibliography}{}
\bibliographystyle{alpha}

\end{document}